%% file: main_ams.tex
\documentclass[11pt,reqno]{amsart}
\usepackage{graphicx} 
    \usepackage[utf8]{inputenc}
    \usepackage{tikz}
\usetikzlibrary{calc,arrows.meta,decorations.pathreplacing}
    \usepackage{amsxtra}
    \usepackage{amssymb}
    \usepackage[margin=1.5in]{geometry}
    \usepackage{comment}
    \usepackage[
    backend=biber,
    style=alphabetic,
    sorting=ynt
    ]{biblatex} 
    \usepackage{amscd}
    \usepackage{amsfonts}
    \usepackage[T1]{fontenc}
    \usepackage{mathtools}
    \usepackage{xparse}
    \usepackage{esint}
    \usepackage{physics}
    \usepackage{dirtytalk}
    
    \usepackage[colorlinks=true, pdfstartview=Fit, linkcolor=blue, citecolor=blue, urlcolor=blue,pagebackref=false]{hyperref}
    \usepackage{microtype}

    \usepackage{bm}
    \usepackage{scalerel} 
    \usepackage{dsfont}
    \usepackage{mathrsfs}

    \usepackage[font={footnotesize}]{caption}
    \usepackage[shortlabels]{enumitem}
    \usepackage{mdwlist}
    \usepackage{wasysym}
    \usepackage{marvosym}
    \newcommand{\SW}{SW}
    \newcommand{\MSW}{\tilde W}

\input{shortcuts}

    \newtheorem{condition}[thm]{Condition}
\hypersetup{pdftitle={Nearly Sharp Comparison Results for sliced and max-sliced Wasserstein distances}, pdfauthor={Jonathan Niles-Weed and Jacob Shkrob}}
\title{Nearly Sharp Comparison Results for sliced and max-sliced Wasserstein distances}
\author[J. Niles-Weed]{Jonathan Niles-Weed$^{1}$}
\author[J. Shkrob]{Jacob Shkrob$^{2}$}
\thanks{$^{1}$Department of Mathematics and Data Science, Courant Institute of Mathematical Sciences, New York
University. \texttt{jnw@cims.nyu.edu}}
\thanks{$^{2}$Courant Institute of Mathematical Sciences, New York University. \texttt{jas10184@nyu.edu}}
\begin{document}

\begin{abstract}
    We prove new comparison results between the Wasserstein distance and its sliced and max-sliced counterparts.
    First, we show that the H\"older exponent~$\frac{2}{d+2}$ obtained by Bobkov and G\"otze for the max-sliced 1-Wasserstein distance on the unit ball is optimal for every $d \geq 2$, settling a question raised in their work.
    Second, we show that sharper comparisons are possible under stronger structural assumptions: if $\nu$ is a discrete measure and the optimal coupling between $\mu$ and $\nu$ transports each point to a nearest atom of $\nu$, then $W_p(\mu, \nu) \leq C \sqrt{d}\, K \, \SW_{p,1}(\mu, \nu)$ for a universal constant $C$, where the complexity parameter $K$ is always at most the number of atoms $N$ and can be substantially smaller.
    This complements a similar bound due to Park and Slepčev.
    An analogous bound holds for the sliced Wasserstein distance based on $k$-dimensional projections.
    Finally, using a construction from geometric discrepancy theory due to Chen and Travaglini, we prove that the linear dependence on $K$ in this bound cannot be improved, up to polylogarithmic factors.
\end{abstract}

\maketitle

\section{Introduction} 
Wasserstein metrics are an important notion of distance between probability distributions, but statistical and computational burdens associated with their use in applications have led to the search for more tractable alternatives.
A prominent option, whose study was initiated by \cite{rabin2011wasserstein, bonneel2015sliced} and \cite{bonnotte2013unidimensional}, is the \textit{sliced Wasserstein distance}, obtained by computing the average Wasserstein distance of one-dimensional projections. For $p \geq 1$ and probability measures $\mu$ and $\nu$ on $\R^d$ with finite $p$th moments, recall that the standard $p$-Wasserstein distance is defined by
\begin{equation}
	W_p(\mu, \nu) := \brac{\inf_{\gamma \in \Gamma(\mu, \nu)} \int_{\R^d \times \R^d} \norm{x - y}^p \dd\gamma(x,y)}^{1/p}\,,
\end{equation}
where $\Gamma(\mu, \nu)$ denotes the set of couplings of $\mu$ and $\nu$, i.e., probability measures on $\R^d \times \R^d$ with marginals $\mu$ and $\nu$.
The sliced Wasserstein distance is given by
\begin{equation}
	\SW_{p,1}(\mu,\nu) := \brac{\int_{\mathbb{S}^{d-1}} W_p^p(\theta_{\#}\mu, \theta_{\#}\nu)\,\,d\sigma_{d}(\theta)}^{1/p},
\end{equation}
where $\theta_\#\mu$ denotes the pushforward of $\mu$ under the map $x \mapsto \abrac{\theta, x}$ and $\sigma_{d}$ is the uniform measure on the sphere $\mathbb{S}^{d-1}$.
Later, applications in statistics and machine learning \cite{kolouri2019generalized, liutkus2019sliced, bonet2022efficient, manole2022minimax, doss2023optimal, rousseau2024wasserstein} suggested studying the \emph{max-sliced Wasserstein distance} \cite{deshpande2019max, paty2019subspace}, with averaging replaced by maximization:
 \begin{equation}
 	\MSW_{p,1}(\mu,\nu) := \max_{\theta \in \mathbb{S}^{d-1}}W_p(\theta_\#\mu,\theta_\# \nu)\,.
 \end{equation}

How do these distances relate to the standard Wasserstein distance?
It is clear that
\begin{equation}\label{eq:trivial_direction}
	\SW_{p, 1}(\mu, \nu) \leq \MSW_{p, 1}(\mu, \nu) \leq W_p(\mu, \nu)
\end{equation}
for all $\mu$ and $\nu$.
In light of~\eqref{eq:trivial_direction}, the key question is to obtain \textit{upper} bounds for $W_p(\mu, \nu)$ in terms of the sliced Wasserstein distance or its max-sliced counterpart.

Simple examples show that $\SW$, $\MSW$, and $W$ are not bi-Lipschitz equivalent, even when restricted to probability measures on a compact subset of $\R^d$ \cite[see, e.g.,][]{bayraktar2021strong, bayraktar2025erratum, park2025geometry, kitagawa2023sliced}.
It is therefore natural to seek quantitative comparisons of H\"older type, that is, bounds of the form
\begin{align}
	W_{p}(\mu, \nu) & \leq C_{p, d, R} \SW_{p, 1}^\alpha(\mu, \nu) \label{sw_comp}\\
	W_{p}(\mu, \nu) & \leq C_{p, d, R}' \MSW_{p, 1}^\beta(\mu, \nu) \label{msw_comp}
\end{align}
for all $\mu, \nu \in \cP(B_R)$, for some $\alpha, \beta > 0$.

The most substantial progress on obtaining such bounds is for $p = 1$, and we focus on this case in the subsequent discussion.
Bonnotte~\cite{bonnotte2013unidimensional} proved the inequality~\eqref{sw_comp} with $\alpha = \tfrac{1}{d+1}$, which stood for over a decade; surprisingly, Carlier, Figalli, M\'erigot, and Wang \cite{carlier2025sharp} very recently showed that~\eqref{sw_comp} holds with $\alpha = \tfrac 1d$ for any $d \geq 3$, and that this is optimal (see~\cite{merigot2026embedding} for an additional introduction).
These questions have a close connection to geometric discrepancy theory: for one-dimensional measures, the Wasserstein distance agrees with the $L^1$ distance between cumulative distribution functions, so bounds on the average half-space discrepancy of point sets~\cite{chen20101} translate directly into bounds on the sliced Wasserstein distance.
We exploit this connection in the proof of Theorem~\ref{thm:discrete_lb} below.

The study of inequality~\eqref{msw_comp} goes back much further than the introduction of the max-sliced Wasserstein distance.
Motivated by the classical Cram\'er--Wold theorem,  Hahn and Quinto~\cite{hahn1985distances} showed that~\eqref{msw_comp} holds with $\beta = 1/\lceil\tfrac{d+3}{2}\rceil$, i.e., $\beta = \tfrac{2}{d+3}$ for odd $d$ and $\beta = \tfrac{2}{d+4}$ for even $d$.\footnote{Their result is stated for the bounded-Lipschitz metric, which is equivalent to $W_1$ on compactly supported measures up to constants depending on the support radius.}
More recently, Bobkov and G\"otze \cite{bobkov2024quantified} proved a stronger bound with $\beta = \tfrac{2}{d+2}$; they also obtain extensions for non-compactly-supported measures.
A simple argument \cite[Section 2]{bobkov2024quantified} shows that~\eqref{msw_comp} can only hold for $\beta \leq \tfrac 2 d$, but, unlike the case of the sliced Wasserstein distance, the sharpest version of~\eqref{msw_comp} was not known.

Our first main result is to show that Bobkov and G\"otze's $\beta = \tfrac{2}{d+2}$ bound is unimprovable.

\begin{thm}\label{thm: sharp exp}
	Let $d \geq 2$. There exist constants $C_d, \varrho_d > 0$ such that, for any $\varrho \in (0, \varrho_d]$, there exist $\mu, \nu \in \cP(B_1)$ satisfying
	\begin{align*}
		W_1(\mu, \nu) & \geq \varrho\\
		\MSW_{1, 1}(\mu, \nu) & \leq C_d \varrho^{1 + d/2}\,.
	\end{align*}
	In particular, for $p = 1$, \eqref{msw_comp} can only hold with $\beta \leq \tfrac{2}{d+2}$.
\end{thm}

Together with prior work, Theorem~\ref{thm: sharp exp} completes the task of finding the best exponents in~\eqref{sw_comp} and~\eqref{msw_comp} when $p = 1$.\footnote{We note that for $p > 1$, the optimal exponents in~\eqref{sw_comp} and~\eqref{msw_comp} remain, in the words of \cite{carlier2025sharp}, ``widely open,'' though several, likely non-sharp, bounds appear in the literature~\cite{terzioglu2026stability,bonnotte2013unidimensional}.}
In both cases, there is a significant gap between the Wasserstein distance and its sliced or max-sliced counterparts which manifests as a ``curse of dimensionality'': as the dimension grows, both the sliced and max-sliced Wasserstein distances are exponentially coarser than the standard Wasserstein distance.
An analogous phenomenon is well known in statistical contexts: the empirical convergence rates of Wasserstein distances suffer from the curse of dimensionality~\cite{chewi2025statistical}, whereas the convergence rate of the max-sliced distance is entirely dimension-free \cite{boedihardjo2025sharp}.

This raises the question of identifying structural assumptions on the measures under which substantially tighter comparisons between $W$ and $\SW$ or $\MSW$ are possible. As a simple example,  Doss,  Wu,  Yang, and Zhou~\cite[Lemma 3.1]{doss2023optimal} show that if $\mu$ and $\nu$ are both supported on at most $N$ atoms, then
\begin{equation}\label{eq:doss}
	W_1(\mu, \nu) \leq \sqrt d N^2 \MSW_{1,1}(\mu, \nu)\,.
\end{equation}
Unlike the general H\"older-type bounds in~\eqref{sw_comp} and~\eqref{msw_comp}, this result gives a Lipschitz-type upper bound on $W_1$, at the price of a multiplicative constant that depends on the support of $\mu$ and $\nu$.

More recently, Park and Slepčev showed \cite{park2025geometry} that $\SW_{2,1}$ and $W_2$ are \textit{locally} comparable in the neighborhood of a discrete measure with respect to the $\infty$-Wasserstein distance \cite{champion2008infinity}, defined by
\begin{equation}\label{eq:infty_def}
	W_\infty(\mu, \nu) := \inf_{\gamma \in \Gamma(\mu, \nu)} \sup_{(x,y) \in \operatorname{supp}(\gamma)} \|x - y\|\,.
\end{equation}
They show that given  a finitely supported probability measure $\nu$ supported on $N$ atoms, each of which is at distance at least $\delta_\nu$ from the others, if $W_\infty(\mu, \nu) < \tfrac{\delta_\nu}{4 C_d N}$, then
\begin{equation}\label{eq:park_slepcev}
	    W_2(\mu,\nu) \leq \sqrt{d}(1 + 4C_dN\delta_\nu^{-1}W_{\infty}(\mu,\nu))^{1/2}SW_{2,1}(\mu,\nu)\,,
\end{equation}
where $C_d \geq 1$ is a constant depending only on the dimension.
Like~\eqref{eq:doss}, this bound still requires that $\nu$ be finitely supported, but obtains a Lipschitz-type comparison under less stringent conditions on $\mu$.

To clarify the scope of Park and Slepčev's bound, we first investigate the assumption that $W_\infty(\mu, \nu)$ is small.
Their requirement that $W_\infty(\mu, \nu) < \tfrac{\delta_\nu}{4 C_d N}$ implies in particular that
the following condition holds.
\begin{condition}\label{cond:voronoi}
	Let $\cY$ be the support of $\nu$.
	There exists a coupling $\gamma \in \Gamma(\mu, \nu)$ under which
	\begin{equation}\label{eq:voronoi}
		\|x - y\| = d(x, \cY) := \min_{y \in \cY} \|x - y\|\,,  \quad \text{for $\gamma$-a.e. $(x, y)$.}
	\end{equation}
\end{condition}
Indeed, the definition of $\delta_\nu$ implies that an open ball of radius $\delta_\nu/2$ can contain at most one element of $\cY$, so if $\|x - y\| < \tfrac{\delta_\nu}{2}$ for some $y \in \cY$, then $y$ is the closest point in $\cY$ to $x$.
Therefore, if $\gamma$ is any optimal coupling in~\eqref{eq:infty_def}\footnote{Such couplings exist by \cite[Section 2]{champion2008infinity}} and $W_\infty(\mu, \nu) < \tfrac{\delta_\nu}{4 C_d N}$, then $\mu$-a.e.\ $x$ is matched to $y \in \cY$ satisfying $\|x - y\| < \tfrac{\delta_\nu}{4 C_d N} < \tfrac{\delta_\nu}{2}$ and hence Condition~\ref{cond:voronoi} holds.

Condition~\ref{cond:voronoi} says that the mass of $\mu$ can be matched entirely within the Voronoi cells induced by $\cY$.
It is easy to see (Lemma~\ref{lem:voronoi_reduction}) that the resulting coupling is automatically optimal for all $p \geq 1$, and therefore evaluating $W_p(\mu, \nu)$ particularly simple.

Our second main result extends Park and Slepčev's bound to show that $W_p$ and $\SW_{p,1}$ are comparable whenever Condition~\ref{cond:voronoi}  holds, with a constant that depends on the complexity of $\nu$ relative to $\mu$.

Given a probability measure $\mu$ and a probability measure $\nu$ supported on a countable closed set $\cY$, define
\begin{equation}\label{eq:k_def}
	K_{\mu,\nu} := \operatorname*{ess\,sup}_{\substack{x \sim \mu \\ x \notin \cY}} \, d(x, \cY) \sum_{y \in \cY} \frac{1}{\|x - y\|}\,,
\end{equation}
where we interpret $K_{\mu, \nu} := 1$ if $\mu(\cY) = 1$.
Note that $K_{\mu, \nu} \geq 1$, since for $x \notin \cY$ the nearest atom alone contributes $1$ to the sum, and that $K_{\mu,\nu}$ depends on $\nu$ only through its support.
Since $d(x, \cY) \leq \|x - y\|$ for every $y \in \cY$, we always have $K_{\mu, \nu} \leq |\cY|$, but $K_{\mu,\nu}$ can be smaller if the points in $\cY$ are well spread relative to the support of $\mu$, see, for example, the construction presented in Section~\ref{sec:lb}.
We prove the following.

\begin{thm}\label{thm:discrete_bound}
	Let $p \geq 1$ and $d \geq 2$, let $\mu \in \cP_p(\R^d)$, and let $\nu$ be a probability measure supported on a countable closed set $\cY$, with $K_{\mu,\nu}$ as in~\eqref{eq:k_def}.
	If Condition~\ref{cond:voronoi} holds, then
	\begin{equation}\label{eq:discrete_bound}
		W_p(\mu, \nu) \leq C \sqrt{d}\, K_{\mu,\nu}\, \SW_{p,1}(\mu, \nu)\,,
	\end{equation}
	where $C$ is a universal constant.
\end{thm}
Compared with the bound appearing in \cite{park2025geometry}, Theorem~\ref{thm:discrete_bound} holds under more general conditions on $\mu$ and $\nu$.
On the other hand, \eqref{eq:park_slepcev} becomes sharper as $W_\infty(\mu, \nu)$ shrinks, whereas the multiplicative constant in~\eqref{eq:discrete_bound} does not improve as $\mu$ approaches $\nu$.
In particular, in the local regime studied by Park and Slepčev, their bound improves on Theorem~\ref{thm:discrete_bound} by a factor of $K_{\mu, \nu}$.
Since $\SW_{p,1} \leq \MSW_{p,1}$, the bound~\eqref{eq:discrete_bound} also holds with the max-sliced distance in place of the sliced distance.

Towards establishing the tightness of Theorem~\ref{thm:discrete_bound}, we prove the following result, showing by example that the dependence on $K_{\mu,\nu}$ in~\eqref{eq:discrete_bound} is unavoidable up to polylogarithmic factors.

\begin{thm}\label{thm:discrete_lb}
	Let $d \geq 2$. For any $K \geq 1$, there exist a finitely supported probability measure $\nu$ and a probability measure $\mu$ on $\R^d$ with $K_{\mu,\nu} \asymp_d K$ such that Condition~\ref{cond:voronoi}  holds and
	\begin{equation}\label{eq:discrete_lb}
		W_1(\mu, \nu) \geq c_{d} \frac{K}{(\log (K+1))^d} \SW_{1,1}(\mu, \nu)\,.
	\end{equation}
\end{thm}

Theorems~\ref{thm:discrete_bound} and~\ref{thm:discrete_lb} show that any Lipschitz-type comparison inequality between the Wasserstein and sliced Wasserstein distances---even under the strong Condition~\ref{cond:voronoi}---requires an almost-linear dependence on the complexity measure $K_{\mu,\nu}$.
The construction involved in the proof of Theorem~\ref{thm:discrete_lb} furnishes a natural example of a setting in which Theorem~\ref{thm:discrete_bound} applies but Park and Slepčev's bound does not.

Finally, we note that the upper bound of Theorem~\ref{thm:discrete_bound} extends with minor modification to a generalization of the sliced Wasserstein distance in which one-dimensional projections are replaced by $k$-dimensional projections.
For $1 \leq k < d$, let $\mathfrak{G}_{d,k}$ denote the Stiefel manifold of matrices $U \in \R^{k \times d}$ with orthonormal rows, equipped with its orthogonally invariant probability measure $\sigma_{d,k}$, and define
\begin{equation}\label{eq:swk_def}
	\SW_{p,k}(\mu,\nu) := \brac{\int_{\mathfrak{G}_{d,k}} W_p^p(U_{\#}\mu, U_{\#}\nu)\,d\sigma_{d,k}(U)}^{1/p},
\end{equation}
where $U_\#\mu$ denotes the distribution of $UX$ when $X \sim \mu$.
(Note that when $k=1$ we recover the sliced Wasserstein distance defined above.)
We define a $k$-dimensional analogue of \eqref{eq:k_def} by
\begin{equation}\label{eq:kk_def}
	K^{(k)}_{\mu,\nu} := \operatorname*{ess\,sup}_{\substack{x \sim \mu \\ x \notin \cY}}\, d(x, \cY) \Big(\sum_{y \in \cY} \frac{1}{\|x - y\|^k}\Big)^{1/k}\,,
\end{equation}
with the same convention as in \eqref{eq:k_def} if $\mu(\cY) = 1$, which satisfies $K^{(1)}_{\mu, \nu} = K_{\mu,\nu}$ and $K^{(k)}_{\mu, \nu} \leq |\cY|^{1/k}$ for all $k$.

\begin{thm}\label{thm:kdim_bound}
	Let $p \geq 1$ and $1 \leq k < d$, let $\mu \in \cP_p(\R^d)$, and let $\nu$ be a probability measure supported on a countable closed set $\cY$.
	If Condition~\ref{cond:voronoi} holds, then
	\begin{equation}\label{eq:kdim_bound}
		W_p(\mu, \nu) \leq C \sqrt{\frac dk}\, K^{(k)}_{\mu,\nu}\, \SW_{p,k}(\mu, \nu)\,,
	\end{equation}
	where $C$ is a universal constant.
\end{thm}
In particular, since $K^{(k)}_{\mu,\nu} \leq N^{1/k}$ when $\nu$ has $N$ atoms, the dependence of the constant on the number of atoms improves from $N$ to $N^{1/k}$ as the dimension of the projections grows.
A similar phenomenon was noted by~\cite{carlier2025sharp}, who proved a H\"older-type comparison for $\SW_{p, k}$ which improves as $k$ grows.

\subsection*{Organization}
Section~\ref{sec: sharp} proves Theorem~\ref{thm: sharp exp}.
Section~\ref{sec:discrete} proves Theorems~\ref{thm:discrete_bound} and~\ref{thm:kdim_bound}, and Section~\ref{sec:lb} proves Theorem~\ref{thm:discrete_lb}.

\subsection*{Notation}
We write $B_R = B_R(0)$ for the closed Euclidean ball of radius $R$ centered at the origin in $\R^d$, $\cP(A)$ for the set of Borel probability measures supported on $A$, and $\cP_p(\R^d)$ for the set of Borel probability measures on $\R^d$ with finite $p$th moment.
For a positive integer $N$, we set $[N] := \{1, \dots, N\}$.
We use $\lesssim$ and $\gtrsim$ to denote inequalities holding up to universal positive constants, with subscripts ($\lesssim_d$, $\lesssim_{d,p}$, etc.) indicating that the constants may depend on the subscripted parameters; $a \asymp b$ means $a \lesssim b \lesssim a$.
The notation $\supp(\mu)$ denotes the support of a measure $\mu$, and $f_\#\mu$ the pushforward of $\mu$ under a map $f$.

\section{Sharpness of the H\"older exponent $\frac{2}{d+2}$ for $\MSW_{1,1}$} \label{sec: sharp}
In this section we prove Theorem \ref{thm: sharp exp}, i.e., the sharpness of the H\"older exponent $\beta = \frac{2}{d + 2}$ left open in \cite{bobkov2024quantified}, for every dimension $d\geq 2$, drawing heavily on the construction in \cite[Section 2.5]{carlier2025sharp}. The heart of the proof is a construction, carried out in Sections~\ref{sec:sharp_prelim} and~\ref{sec:sharp_proof}, of a pair of measures $(\mu_\epsilon, \nu_\epsilon)$ for every sufficiently small $\epsilon > 0$, each supported in $B_1(0)$, such that
\begin{equation}\label{eq: contrad scale}
	W_1(\mu_\epsilon, \nu_\epsilon) \geq \kappa_d\, \epsilon^{2-2/d} \quad \text{and} \quad \MSW_{1,1}(\mu_\epsilon, \nu_\epsilon) \leq \kappa_d^{\prime}\, \epsilon^{(d+1) - 2/d}
\end{equation}
for constants $\kappa_d, \kappa_d' > 0$.

Let us first check that this construction proves Theorem~\ref{thm: sharp exp}.
Given $\varrho > 0$ sufficiently small, choose $\epsilon$ so that $\kappa_d \epsilon^{2 - 2/d} = \varrho$.
Since
$$
	\frac{(d+1) - 2/d}{2 - 2/d} = \frac{d^2 + d - 2}{2d - 2} = \frac{(d+2)(d-1)}{2(d-1)} = \frac{d+2}{2} = 1 + \frac d2\,,
$$
the pair $(\mu, \nu) = (\mu_\epsilon, \nu_\epsilon)$ satisfies $W_1(\mu, \nu) \geq \varrho$ and
$$
	\MSW_{1,1}(\mu, \nu) \leq \kappa_d' \brac{\varrho/\kappa_d}^{1 + d/2} = C_d\, \varrho^{1 + d/2}\,,
$$
as claimed.
For the final claim of the theorem, suppose that~\eqref{msw_comp} holds for some $\beta > 0$, i.e., $W_1(\mu,\nu) \leq C \MSW_{1,1}^\beta(\mu,\nu)$ uniformly over $\cP(B_1)$. Applying this bound to $(\mu_\epsilon, \nu_\epsilon)$ yields
$$
	\varrho \leq W_1(\mu_\epsilon, \nu_\epsilon) \leq C \MSW_{1,1}^\beta(\mu_\epsilon, \nu_\epsilon) \leq C C_d^\beta\, \varrho^{\beta(1 + d/2)}\,,
$$
and letting $\varrho \to 0$ forces $\beta(1 + d/2) \leq 1$, i.e., $\beta \leq \frac{2}{d+2}$.

We first describe the idea behind the construction and compare it with the approach due to Carlier, Figalli, M\'erigot, and Wang~\cite{carlier2025sharp}. The construction proving the sharpness of the exponent $\alpha = \frac{1}{d}$ for $SW_{1,1}$ in \cite{carlier2025sharp} takes a pair $(\mu,\mu_\epsilon)$ where $\mu_\epsilon$ is a perturbation at scale $\epsilon$ in the direction $e_d$ of a $(d-1)$-dimensional measure $\mu$ supported on $e_d^{\perp}$; such a pair is far in $W_1$ but the perturbation is visible only for the small fraction of directions $\theta$ nearly aligned with $e_d$. Since the max-sliced distance takes a supremum over directions rather than an average, a single adversarial direction no longer suffices; our construction instead superimposes perturbations along a spread-out collection of $m\sim \epsilon^{-(d-1)}$ directions given by an $\epsilon$-packing of the sphere, so that every direction $\theta$ sees only a small perturbation.

\subsection{Heuristics: stars of Gaussian pancakes}\label{sec:pancakes}
In this subsection we present a heuristic version of the construction, based on \say{Gaussian pancakes}, a popular adversarial construction in learning theory \cite{bubeck2019adversarial, diakonikolas2017statistical, bruna2021continuous}.
Nothing in this subsection is used in the sequel; its purpose is to motivate the choices made in Sections~\ref{sec:sharp_prelim} and~\ref{sec:sharp_proof}, where the actual construction is carried out, and we accordingly argue without full proofs.

Given a direction $u \in \mathbb{S}^{d-1}$, a radius $r \in (0, 1]$, and a thickness $0 < \rho \leq r$, we construct two Gaussian mixtures $\gamma_u^+$ and $\gamma_u^-$ whose means lie in the subspace spanned by $u$.
Each component of this mixture is a Gaussian measure with covariance  $r^2 P_{u^\perp} + \rho^2 u u^\top$, where $P_{u^\perp}$ is the orthogonal projection onto $u^\perp$: these are the \say{pancakes} of extent $r$ in the directions of $u^\perp$ and thickness $\rho$ along its normal $u$.
The means have norm at most $\rho$, and they are chosen such that the first $d+1$ moments of the two mixtures agree and $W_1(\gamma_u^+, \gamma_u^-) \asymp \rho$.
(In the rigorous version of this construction, it is convenient to specify the signed measure $\gamma_u = \gamma_u^+ - \gamma_u^-$ directly and analyze the Kantorovich--Rubinstein norm, see Definition~\ref{def: 1 homg}).

What does a one-dimensional projection see of this perturbation? Fix $\theta \in \mathbb{S}^{d-1}$ and write $\alpha := \abrac{\theta, u}$ and $\beta := (1-\alpha^2)^{1/2}$.
The measures $\gamma_u^+$ and $\gamma_u^-$ differ only in the $u$ direction.
After projecting along $\theta$, the resulting measures are two different Gaussian mixtures (arising from the perturbation in the $u$ direction) convolved with Gaussian noise of variance $r^2 \beta^2$ (arising from the directions orthogonal to $u$).

It is well known (see, e.g., \cite{wu2020optimal, doss2023optimal}) that the smoothing effects of the heat semigroup imply strong bounds on the Wasserstein distance between Gaussian mixtures whose moments match; using the fact that the first $d+1$ moments of the mixtures agree, one can show that the dominant contribution is of order $(d+2)$:
\begin{equation}\label{eq:pancake_kernel}
	W_1(\theta_\#\gamma_u^+, \theta_\#\gamma_u^-) \lesssim_{d} \min\sbrac{\rho,\; \frac{(\rho|\alpha|)^{d+2}}{(r\beta)^{d+1}}} \leq r \min\sbrac{\epsilon,\; \frac{\epsilon^{d+2}}{\beta^{d+1}}}, \qquad \epsilon := \frac \rho r\,.
\end{equation}
The parameter $\epsilon = \rho/r$ is the \emph{eccentricity} of the pancake, and \eqref{eq:pancake_kernel} says that
 the perturbation is visible at full strength $\rho$ only when the angle between $u$ and $\theta$ is at most $\epsilon$, with rapid decay beyond that scale.
 
The above considerations, applied to a single direction $u$, suffice to establish the sharp bound for the sliced Wasserstein distance due to~\cite{carlier2025sharp}.
To prove a stronger bound for the max-sliced distance, we superimpose these pancakes over a spread-out family of normals: let $\sbrac{u_i}_{i \in [m]}$ be a maximal $\epsilon$-separated subset of $\R\prob^{d-1}$, so that $m \asymp \epsilon^{-(d-1)}$ (Lemma~\ref{lem: packing}), place a pair of pancake mixtures $\gamma^\pm_{u_i}$ with radius $r$ and thickness $\rho = r\epsilon$ in direction $u_i$, and take
$$
	\mu = \frac 1m \sum_{i=1}^m \gamma^+_{u_i}, \qquad \nu = \frac 1m \sum_{i=1}^m \gamma_{u_i}^{-}\,.
$$
Any fixed $\theta$ is within angle $\epsilon$ of only $O(1)$ of the directions.
Those aligned directions contribute $\rho/m$ each to the projected distance, and the contributions decay sharply away from $\theta$, so that the contribution from well aligned directions dominates, yielding
\begin{equation}\label{eq:pancake_ub}
	\MSW_{1,1}(\mu, \nu) \lesssim
	 \frac{\rho}{m} \asymp \rho\, \epsilon^{d-1}\,.
\end{equation}
Meanwhile, since $W_1(\gamma_u^+, \gamma_u^-)\asymp \rho$, we anticipate that  $W_1(\mu, \nu) \asymp \rho$ as well.
This will only hold if there is not significant interaction between the pancakes in different directions; more precisely, we can require that the supports of the mixtures in each direction are essentially disjoint.
Up to Gaussian tails, each pancake mixture occupies a disk of radius $r$ and thickness $\rho$, so its support has volume $\asymp r^{d-1}\rho$.
Requiring that these sets be essentially disjoint forces the sparsity constraint
\begin{equation}\label{eq:sparsity}
	m \rho\, r^{d-1} \lesssim 1\,.
\end{equation}

It remains to optimize.
Since $m \asymp \epsilon^{1-d} = (r/\rho)^{d-1}$, the sparsity constraint~\eqref{eq:sparsity} imposes $r \lesssim \rho^{\frac{d-2}{2(d-1)}}$.
Choosing $r$ as large as possible subject to this constraint gives $\epsilon = \rho/r \asymp \rho^{\frac{d}{2(d-1)}}$, so that $\epsilon^{d-1} \asymp \rho^{d/2}$ and
$$
	W_1(\mu, \nu) \asymp \rho, \qquad \MSW_{1,1}(\mu, \nu) \lesssim \rho^{1 + d/2}\,,
$$
which gives the bound claimed in Theorem~\ref{thm: sharp exp}.
Reparametrizing by the eccentricity $\epsilon$ recovers the scalings used in Section~\ref{sec:sharp_prelim}:
$$
	r = a\epsilon^{(d-2)/d}, \qquad \rho = a \epsilon^{2 - 2/d}, \qquad m \asymp \epsilon^{-(d-1)}\,.
$$
A version of this construction when $d = 2$ is depicted in Figure~\ref{fig:pancakes}.\footnote{Note that in $d = 2$, it is possible to take $r \asymp 1$, but in $d \geq 3$ the pancakes must shrink ($r \to 0$).}

To summarize, the rigorous construction of Sections~\ref{sec:sharp_prelim} and~\ref{sec:sharp_proof} implements the above blueprint with three modifications.
First, since Theorem~\ref{thm: sharp exp} concerns measures on $B_1$, the Gaussian measures are replaced by the compactly supported densities $(1 - \norm{y - z}^2/r^2)_+^{d+1}$ on a disk of radius $r$, whose projections are smooth enough to mimic the heat semigroup argument described above (see Lemma~\ref{lem: upper w_1}).
Second, the measures constructed in the general case contain a singular part supported on a union of $(d-1)$-dimensional disks; this simplifies the proof of the lower bound on $W_1$.
Finally, in addition to being oriented in different directions, the component measures in the rigorous construction are randomly shifted.
This step ensures that the supports of the mixtures in different directions are essentially disjoint, so that the necessary condition~\eqref{eq:sparsity} is also sufficient.

\begin{figure}[t]
    \centering
    \resizebox{0.8\textwidth}{!}{\input{figures/sliced_diagram}}
    \caption{The star of Gaussian pancakes in the illustrative case $d = 2$: the mixtures $\mu$ (left) and $\nu$ (right)
    and their one-dimensional projections
    $\theta_\#\mu$ and $\theta_\#\nu$ along a direction $\theta$. (For legibility, the perturbation depicted is a symmetric two-point shift of each pancake, which matches only the first moment.)}
    \label{fig:pancakes}
\end{figure}
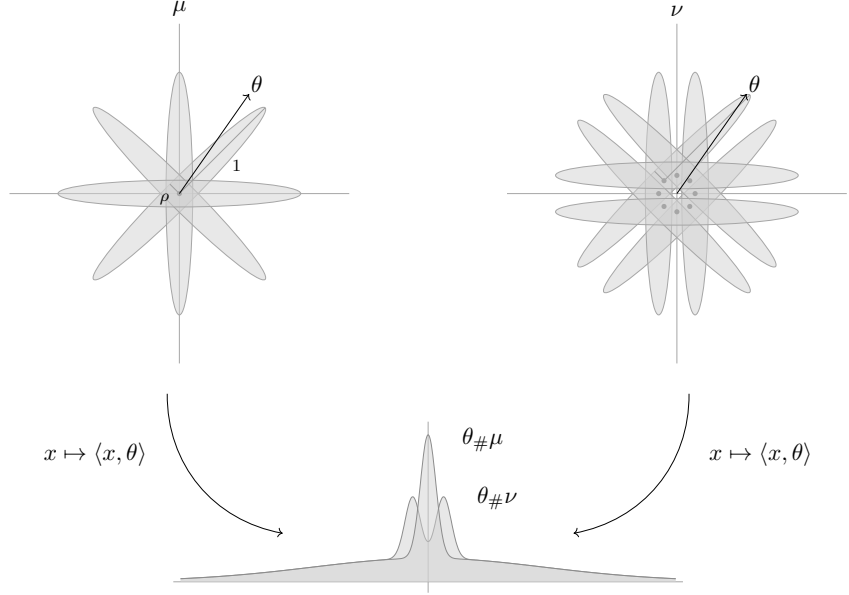

\subsection{Steps of the general construction}\label{sec:sharp_prelim}
Throughout the remainder of this section we fix $d \geq 2$.
As alluded to above, our aim is to construct two one-dimensional perturbations whose moments match, and embed these one-dimensional perturbations densely on the unit sphere.
As in~\cite{carlier2025sharp}, it is convenient to replace the Wasserstein distance with the Kantorovich--Rubinstein norm and work directly with signed measures.
\begin{definition}[Kantorovich--Rubinstein norm]\label{def: 1 homg}
		Let $\sigma$ be a signed measure on $\R^d$ with $\sigma(\R^d) = 0$.
    The Kantorovich--Rubinstein norm of $\sigma$ is
     \begin{equation}
    	\norm{\sigma}_{\text{KR}} := \sup_{f \in \text{Lip}_1(\R^d,\R)}\int f \, d \sigma\,.
    \end{equation}
\end{definition}
Standard Kantorovich duality \cite[Theorem 1.14]{villani2003topics} shows that $W_1(\mu, \nu) = \norm{\mu - \nu}_{\text{KR}}$.
We therefore use the latter quantity as the definition of $W_1(\mu, \nu)$ for general measures of the same total mass.

The remainder of this subsection follows the outline of Section~\ref{sec:pancakes}.
We first construct the moment-matched perturbation associated with a single direction and prove the smoothing bound analogous to  the heat semigroup estimate (Lemma~\ref{lem: upper w_1}).
We then introduce the packing of the space of directions (Lemma~\ref{lem: packing}), assemble the construction with the parameters found in Section~\ref{sec:pancakes}, and choose the centers so as to realize the sparsity constraint \eqref{eq:sparsity} (Lemma~\ref{lem: estimate}).
Lemma~\ref{lem: norm const} records the normalization needed to convert the resulting signed measure into a pair of probability measures.

\subsubsection*{One-dimensional mixtures}
We construct the rigorous counterpart of the pair $(\gamma_u^+, \gamma_u^-)$ of Section~\ref{sec:pancakes}.

As in \cite{carlier2025sharp}, define $g \in C_c^\infty((-1,1))$ to be a (signed) weight function with the following properties: for every $j \in \sbrac{1,2,\dots, d+1}$,
\begin{equation}\label{eq: g constr}
    \int_{-1}^1 g(s)\,ds = 1, \quad \int_{-1}^1 s^j g(s)\,ds = 0, \quad \int_{-1}^1 |s|g(s) \,ds =: c_g > 0.
\end{equation}
(Such a function exists by the linear independence of $1, s, s^2, \dots, s^{d+1}, |s|$.)
Next for given datum $(u,z,r) \in \mathbb{S}^{d-1}\times \R^d \times \R_{> 0}$, define the shifted $(d-1)$-dimensional ball $D(u,z,r)$ 
\begin{equation}
    D(u,z,r) := z + \{y \in u^\perp: \norm{y}\leq r\}.
\end{equation}
We assign to each datum $(u,z,r)$ a probability measure $\la_{(u,z,r)}$ supported on $D(u,z,r)$, up to a normalizing constant depending on $u,z,r,$ and $d$,
\begin{equation}
    \la_{(u,z,r)}(y) \propto (1 - {\norm{y-z}^2}/{r^2})_+^{d+1}\indic_{D(u,z,r)}(y)\,d\cH^{d-1}(y),
\end{equation}
where $\cH^{d-1}$ is the standard $(d-1)$-dimensional Hausdorff measure. As in \cite{carlier2025sharp}, define the shift operation $\tau_{v}(x) = x+v$. For a given $\epsilon > 0$ and $r > 0$, we define the following $g$-weighted superposition of translates of $\la_{(u,z,r)}$:
\begin{equation}
    \la^{(r\epsilon)}_{(u,z,r)} := \int_{-1}^1 g(s) (\tau_{s r\epsilon u})_\# \la_{(u,z,r)}\,ds. 
\end{equation}
Notice that $\la^{(r\epsilon)}_{(u,z,r)}$ is a signed measure of total mass $\la^{(r\epsilon)}_{(u,z,r)}(\R^d) = 1$, supported on a thin cylinder with radius $r$ in $(d-1)$ dimensions and height $2r\epsilon$ in the direction of $u$; since $g$ takes negative values, $\la^{(r\epsilon)}_{(u,z,r)}$ has a nontrivial negative part.
The signed measure $\la^{(r\epsilon)}_{(u,z,r)} - \la_{(u,z,r)}$ plays the role of $\gamma_u^+ - \gamma_u^-$: it has total mass zero, and, by \eqref{eq: g constr}, its moments in the direction $u$ vanish up to order $d+1$.

The next lemma is the analogue of~\eqref{eq:pancake_kernel}, showing that the perturbation is visible only when $\theta$ is at an angle of roughly $\epsilon$ from the normal $u$.

\begin{lem}\label{lem: upper w_1}
Fix $\epsilon > 0$. For every $\theta \in \mathbb{S}^{d-1}$ and datum $(u,z,r) \in \mathbb{S}^{d-1}\times \R^d \times \R_{> 0}$ with $\rho = r\epsilon$,
\begin{equation}
    \lVert{\theta_\#(\la^{(\rho)}_{(u,z,r)}-\la_{(u,z,r)})\rVert}_{\mathrm{KR}} \leq C_{d,g} r \min\sbrac{\epsilon, \frac{\epsilon^{d+2}}{\beta^{d+1}}},
\end{equation}
where $\beta := (1 - \abrac{u,\theta}^2)^{1/2}$ and the second term in the minimum is interpreted as $+\infty$ when $\beta = 0$.
\end{lem}
\begin{proof}
    Since $\lVert \theta_\# (\la^{(\rho)}_{(u,z,r)} - \la_{(u,z,r)})\rVert_{\mathrm{KR}}$ does not depend on $z \in \R^d$, without loss of generality let $z = 0$. In this case, we set
    $$\la^{(\rho)}_{(u,z,r)}=\la^{(\rho)}_{(u,r)}, \quad  \la_{(u,z,r)} = \la_{(u,r)}$$
    for simplicity. A crude bound gives, irrespective of $\theta$,
    \begin{equation*}
        \lVert \theta_\# (\la^{(\rho)}_{(u,r)} - \la_{(u,r)})\rVert_{\mathrm{KR}} \leq \int_{-1}^1 |g(s)|\lVert \theta_{\#}((\tau_{s\rho u})_{\#}\la_{(u,r)} - \la_{(u,r)})\rVert_{\mathrm{KR}} \,ds \leq C_g\rho = C_g r\epsilon,
    \end{equation*}
since translating a measure by $s\rho u$ moves its projection by at most $|s|\rho$. This proves the lemma when $\beta \leq \epsilon$, so assume from now on that $\beta > \epsilon$.

    As in \cite{carlier2025sharp}, we decompose $\theta$ along the hyperplane $u^\perp$: setting $\alpha := \abrac{u,\theta}$, so that $\alpha^2 + \beta^2 = 1$, we may write
    $$
        \theta = \beta\,\omega + \alpha\,u, \qquad \omega := \frac{\theta - \alpha u}{\beta} \in u^\perp \cap \mathbb{S}^{d-1}.
    $$
    For $x \in u^\perp$, we have $\abrac{\theta, x} = \beta\abrac{\omega, x}$, and so, recalling that $\la_{(u,r)}$ is the dilation by $r$ of $\la_{(u,1)}$,
    \begin{equation}
        \theta_{\#}\la_{(u,r)} = (t\mapsto r\beta\,t)_\#(\omega_{\#}\la_{(u,1)})\,.
    \end{equation}

    Let $f_1$ denote the density of $\omega_{\#}\la_{(u,1)}$ and $F_1$ its cumulative distribution function; by the rotational invariance of $\la_{(u,1)}$ within $u^\perp$, these do not depend on $\omega$, and integrating out the $d-2$ orthogonal directions gives the explicit form
    $$
	    f_1(t) = c_d (1 - t^2)_+^{3d/2}\,.
    $$
    Since $3d/2 \geq d+1$ for $d \geq 2$, the function $f_1$ is of class $C^{d}$ on $\R$, its $d$th derivative is absolutely continuous, and
    $$
	    |f_1^{(d+1)}(t)| \lesssim_d (1-t^2)_+^{d/2 - 1} \cdot \indic_{[-1,1]}(t) \implies \|f_1^{(d+1)}\|_{L^1(\R)} \leq C_d\,.
    $$

    Furthermore, translating by $s \rho u$ shifts the projection onto $\theta$ by $s\rho\abrac{\theta, u} = s\rho \alpha$, so the cumulative distribution function of $\theta_\#\la^{(\rho)}_{(u,r)}$ is $t \mapsto \int_{-1}^1 g(s) F_1\brac{\frac{t - s\rho \alpha}{r\beta}} ds$.
    By the $L_1$ representation of $W_1$ in one dimension \cite{bobkov2019one}, followed by the substitution $t = r\beta\,\tau$,
    \begin{align*}
        \lVert \theta_\#(\la^{(\rho)}_{(u,r)} - \la_{(u,r)})\rVert_{\textrm{KR}} &= \int_{-\infty}^\infty \abs{\int_{-1}^1 g(s) \brac{F_1\brac{\frac{t-s\rho \alpha}{r\beta}} - F_1\brac{\frac{t}{r\beta}}} \,ds}\,dt\\
        &= r\beta\int_{-\infty}^\infty\abs{\int_{-1}^1 g(s)\brac{F_{1}(\tau-sh) - F_{1}(\tau)}\,ds}\,d\tau, \qquad h := \frac{\rho \alpha}{r\beta} = \frac{\epsilon \alpha}{\beta}\,.
    \end{align*}

    We expand $F_1$ to order $d+1$ with integral remainder:
    \begin{equation}
	    F_1(\tau - sh) = F_1(\tau) + \sum_{j=1}^{d+1}\frac{F_{1}^{(j)}(\tau)}{j!}(-sh)^{j} + \frac{(-sh)^{d+2}}{(d+1)!}\int_0^1 (1-v)^{d+1} F_1^{(d+2)}(\tau - vsh)\,dv\,.
    \end{equation}
    By our choice of $g$ in \eqref{eq: g constr}, integrating against $g(s)\,ds$ kills the zeroth through $(d+1)$st-order terms, so that by Fubini--Tonelli
    \begin{align}
        \lVert \theta_\#(\la^{(\rho)}_{(u,r)} - \la_{(u,r)})\rVert_{\textrm{KR}}
	    &\leq r\beta \cdot \frac{|h|^{d+2}}{(d+1)!} \int_{-1}^1 |g(s)||s|^{d+2} \,ds \cdot \sup_{w \in \R}\, \|F_1^{(d+2)}(\cdot - w)\|_{L^1(\R)} \nonumber \\
	    &\leq C_{d,g}\, r \beta \abs{\frac{\epsilon \alpha}{\beta}}^{d+2} \leq \frac{C_{d,g}r \epsilon^{d+2}}{\beta^{d+1}}\,, \label{eq: final bound}
    \end{align}
    where we used $\|F_1^{(d+2)}\|_{L^1} = \|f_1^{(d+1)}\|_{L^1} \leq C_d$ and $|\alpha| \leq 1$.

    Combining \eqref{eq: final bound} with the crude bound gives the result.
\end{proof}

\subsubsection*{Constructing the packing}
We next introduce the family of normal directions along which the perturbations will be superimposed.
The natural home for these directions is the real projective space $\mathbb{RP}^{d-1} = \mathbb{S}^{d-1}/\sim$ where $u \sim v$ if $v = \pm u$. We identify $u\in \mathbb{S}^{d-1}$ with the equivalence class $[u]\in\mathbb{R}\mathbb{P}^{d-1}$ equipped with the projective metric
$$
    \delta([u],[v]) = \sqrt{1-\abrac{u,v}^2}\,.
$$
The following is a classical result for the metric entropy of the Grassmannian going back to Pajor \cite{pajor1998metric} (who attributes the result to Szarek \cite{szarek1982nets}; see also \cite{szarek1997metric}).
\begin{lem}[Local entropy bounds for $(\mathbb{RP}^{d-1},\delta)$]\label{lem: packing} For $0 < \epsilon < 1$, there exists an $\epsilon$-packing of $\R\prob^{d-1}$ of size $m$ with respect to the metric $\delta$
\begin{equation}
    \cU_\epsilon = \sbrac{[u_1],\dots, [u_m]} \subseteq \R\prob^{d-1}
\end{equation}
such that for two universal constants $c, C > 0$,
\begin{equation}\label{eq: metric entrop}
    \brac{\frac{c}{\epsilon}}^{(d-1)} \leq m \leq \brac{\frac{C}{\epsilon}}^{(d-1)}.
\end{equation}
Additionally, for some universal $\tilde{c} > 0$, for each $v\in \R\prob^{d-1}$ and $R \geq \epsilon$,
\begin{equation}\label{eq: local entropy}
    |\sbrac{j \in [m]: \delta([u_j], [v])\leq R}| \leq \brac{\frac{\tilde{c} R}{\epsilon}}^{(d-1)}.
\end{equation}
\end{lem}
\begin{proof}
    The inequality in \eqref{eq: metric entrop} follows directly from the fact that $\R\prob^{d-1}$ is a compact $(d-1)$-dimensional space, see \cite[Proposition 8]{pajor1998metric}. To show the local bound \eqref{eq: local entropy}, if $I_{v,R} = \sbrac{j \in [m]: \delta([u_j],[v]) \leq R}$ and $\cU_\epsilon$ is an $\epsilon$-packing, then the balls $B_\delta([u_j], \epsilon/2)$ for $j \in I_{v,R}$ are pairwise disjoint and contained in $B_\delta([v], R + \epsilon/2)$, so by a volume argument
    $$
        |I_{v,R}| \min_{j}\bar\sigma_{d}(B_{\delta}([u_j], \epsilon/2)) \leq \bar\sigma_{d}(B_\delta([v], R + \epsilon/2)),
    $$
    where $\bar\sigma_{d}$ is the unique probability measure on $\R\prob^{d-1}$ invariant under orthogonal transformations. By Corollary 2.2 of \cite{paouris2014neighborhoods}, $(t/c_1)^{d-1} \leq \bar\sigma_{d}(B_\delta([x],t)) \leq (t/c_2)^{d-1}$ for some $c_1, c_2 > 0$, and so since $R \geq \epsilon$
    $$
        |I_{v,R}| \leq \brac{\frac{\tilde{c}(R + \epsilon/2)}{\epsilon/2}}^{d-1} \leq \brac{\frac{3\tilde{c}R}{\epsilon}}^{d-1},
    $$
    for some constant $\tilde{c} > 0$.
\end{proof}

We now superimpose the perturbations using the packing constructed above.
As in Section~\ref{sec:pancakes}, the eccentricity $\rho/r$ of each slab is chosen to match the separation $\epsilon$ between directions, and $r$ is taken as large as the sparsity constraint allows.
For our construction, we will use the following rescaled parameters: for a constant ${a} > 0$ to be chosen, set
\begin{equation*}
    r := a \epsilon^{(d-2)/d}
\end{equation*}
to be the scale of the radius in $D(u,z,r)$ and 
$$\rho:= r\epsilon = a\epsilon^{2-2/d}$$ to be the scale of the shift in $\la^{(\rho)}_{(u,z,r)}$. 
Let $c > 0$ denote the universal constant appearing in the lower bound of \eqref{eq: metric entrop}, and let $\cU=\sbrac{u_i: i\in [m]}$ be an $\epsilon$-packing of $(\R\prob^{d-1},\delta)$ of size $m = \floor{c^{d-1}\epsilon^{-(d-1)}}$.
Let $\cZ = \sbrac{z_i: i \in [m]}$ be a set of shifts in $B_{1/4}(0)$, which will be chosen as in Lemma \ref{lem: estimate} below. The constant $a$ will be chosen at the end of the proof, depending only on $d$ and $g$; we then take $\epsilon$ sufficiently small depending on $a$.

 For each datum $(u_i,z_i,r)$, define $P_i := D(u_i, z_i, r)$ and $\cE^{(\rho)}$ to be the union of $\rho$-slabs
\begin{equation}
    \cE^{(\rho)} := \bigcup_{j=1}^m P_j + [-\rho,\rho]u_j.
\end{equation}
Define the signed measure
\begin{equation}
    \sigma_\epsilon := \frac{1}{m}\sum_{i=1}^m (\la^{(\rho)}_{(u_i, z_i, r)} - \la_{(u_i,z_i, r)}).
\end{equation}
For every $i \in [m]$, $\la_{(u_i,z_i,r)}(\R^d\setminus P_i) = 0$ and since $\epsilon \in (0,1)$, we also have  $\supp(\la^{(\rho)}_{(u_i,z_i,r)}) \subseteq \cE^{(\rho)}$.
Notice that $\sigma_\epsilon$ is a nonzero signed measure satisfying $\sigma_\epsilon(\R^d) = 0,$ and hence, setting $M_\epsilon := |\sigma_\epsilon^+|,$ we can define probability measures
\begin{align}
    \mu_\epsilon := \frac{\sigma^+_\epsilon}{M_\epsilon}, \quad \nu_\epsilon := \frac{\sigma^{-}_\epsilon}{M_\epsilon}\,.
\end{align}
Note that $\cE^{(\rho)}\subseteq B_1(0)$, since for any $s \in [-1,1]$ and $j \in [m]$, if $y \in B_r(0) \cap u_j^\perp$, by our choice of $\rho$ and $r$,
\begin{equation}
	\norm{z_j + y + s \rho u_j} \leq \frac{1}{4}+ |s|\rho + r \leq 1.
\end{equation}
Therefore $\mu_\epsilon, \nu_\epsilon \in \cP(B_1(0))$.
In Section~\ref{sec:sharp_proof} we will show that they achieve the sharp scale separations in \eqref{eq: contrad scale}, i.e., that $W_1(\mu_\epsilon, \nu_\epsilon)\geq \kappa_d \epsilon^{2-2/d}$ and $\MSW_{1,1}(\mu_\epsilon, \nu_\epsilon) \leq \kappa^{\prime}_d\epsilon^{(d+1)-2/d}$.

\subsubsection*{Constructing the shifts}
We now choose $\cZ$ to ensure that the perturbations planted in different directions do not interact.
The following result shows that when $m \rho r^{d-1}$ is smaller than a sufficiently small constant, the distance of a point $x$ in the support of $\la^{(\rho)}_{(u_i, z_i, r)}$ to $P_i$ is close to the distance from $x$ to the union $\bigcup_{j=1}^m P_j$; in other words, this guarantees that typical points associated with the direction $u_i$ are not close to the mixtures associated with any other direction.

\begin{lem}\label{lem: estimate}
Let $\Phi(x) = d(x,\bigcup_{j=1}^m P_j)$ be the distance function and let $0 < \rho < r$. For any collection $\cU = \sbrac{u_i: i\in[m]}\subseteq \mathbb{S}^{d-1}$, there exists a set $\cZ = \sbrac{z_i: i \in [m]} \subseteq B_{1/4}(0)$ such that
    \begin{equation}\label{eq: dist bound}
        \cD_{\cZ} := \frac{1}{m}\sum_{j=1}^m\int_{\R^d} \int_{-1}^1  |g(s)||\Phi(x + s\rho u_j) - \rho|s||\,ds\,d\la_{(u_j,z_j,r)} \leq C_{g,d} m\rho^2 r^{d-1}
    \end{equation}
    where $C_{g,d}$ is a constant depending on $g$ and $d$.
\end{lem}
\begin{proof}
To prove \eqref{eq: dist bound}, we proceed with a probabilistic argument. Let $\cZ$ be an $\iid$ random sample $z_1,\dots,z_m$ from $\Unif(B_{1/4}(0))$. Notice that $\Phi(x+s\rho u_j) \leq |s|\rho,$ for any $x \in \cup_{j=1}^m P_j$, and furthermore,
\begin{equation}
    (|s|\rho - \Phi(x+s\rho u_j))_+\leq |s|\rho\sum_{i\neq j}\indic_{\sbrac{d(x+s\rho u_j,P_i)\leq |s|\rho}},
\end{equation}
so it suffices to control for each $j\in [m]$ and $x = z_j + y$ where $y\in B_r(0) \cap u_j^\perp$ the pairwise probability for $i\neq j$
\begin{equation}\label{eq: prob factor}
    \Psub{\cZ}{d\brac{z_j + y + s\rho u_j, z_i + (B_r(0) \cap u_i^\perp)}\leq |s|\rho}.
\end{equation}
The random variable $z_i - z_j$ has a bounded density $f_d$; that is, there exists a constant $C_d > 0$ such that $\norm{f_d}_\infty \leq C_d$. By shift-invariance, \eqref{eq: prob factor} can be rewritten as
\begin{align*}
    \Psub{\cZ}{d\brac{z_j - z_i, (B_r(0) \cap u_i^\perp)-y-s\rho u_j} \leq |s|\rho} &\leq C_d |\{x: d(x, B_r(0)\cap u^{\perp}_i) \leq |s|\rho\}| \\
    &= C_d|\{x: x \in (B_r(0) \cap u_i^\perp) + B_{|s|\rho}(0)\}|\\
    &\leq C_d r^{d-1}|s|\rho,
\end{align*}
where in the last inequality we used the containment $$(B_r(0)\cap u_i^\perp) + B_{|s|\rho}(0) \subseteq (B_{r+|s|\rho}(0) \cap u_i^\perp) \times [-|s|\rho, |s|\rho]$$ to bound the volume. Therefore, summing over all $i\in [m]$ with $i \neq j$ and integrating against $|g(s)|\,ds$,
\begin{equation}
    \E{\cD_{\cZ}} \leq m C_d \tilde{c}_g \rho^2 r^{d-1}, \qquad \tilde c_g := \int_{-1}^1 s^2 |g(s)|\,ds,
\end{equation}
and we can conclude by the probabilistic method the existence of a fixed set $\cZ$ satisfying \eqref{eq: dist bound} with $C_{g,d} = C_d \tilde{c}_g$.
\end{proof}

\subsubsection*{Normalization}
Our last lemma records the following bounds on the size of the normalizing constant $M_\epsilon$ appearing in the definition of $(\mu_\epsilon, \nu_\epsilon)$; it guarantees that passing between the signed measure $\sigma_\epsilon$ and the pair of probability measures $(\mu_\epsilon, \nu_\epsilon)$ costs only a constant factor.
\begin{lem}\label{lem: norm const}
    For $\sigma_\epsilon$ and $M_\epsilon$ defined as before, we have that for a constant $C_g > 0$ depending only on $g$,
    \begin{equation}
        1 \leq M_\epsilon \leq C_g.
    \end{equation}
\end{lem}
\begin{proof}
    Since $\sigma_\epsilon(\R^d) = 0$, we have $|\sigma_\epsilon^+| = |\sigma_\epsilon^-| = M_\epsilon$.
    To show the upper bound, by the triangle inequality for signed measures,
    \begin{align}
        M_\epsilon &\leq \frac{1}{m}\sum_{i=1}^m |(\la^{(\rho)}_{(u_i,z_i,r)} - \la_{(u_i,z_i,r)})^+| \nonumber\\
        &\leq \frac{1}{m}\sum_{i=1}^m \brac{\la_{(u_i,z_i,r)}(\R^d) + |\la^{(\rho)}_{(u_i,z_i,r)}|(\R^d)} \leq 1 + \norm{g}_{L^1([-1,1])} = C_g.
    \end{align}
    For the lower bound, we claim that $\la^{(\rho)}_{(u_i,z_i,r)}(\cup_{j=1}^m P_j) = 0$ for every $i \in [m]$. Indeed, for $s \neq 0$, the translated disk $P_i + s\rho u_i$ is disjoint from $P_i$ (the translation is orthogonal to $P_i$), and for $j \neq i$ it meets the disk $P_j$ in a set of dimension at most $d - 2$ (the disks are not parallel, since $[u_i] \neq [u_j]$), which is $(\tau_{s\rho u_i})_\#\la_{(u_i,z_i,r)}$-null. Since $s = 0$ is Lebesgue-null in the $ds$ integral defining $\la^{(\rho)}_{(u_i,z_i,r)}$, the claim follows. On the other hand, $\frac 1m \sum_{i} \la_{(u_i,z_i,r)}(\cup_j P_j) = 1$, so the restriction of $\sigma_\epsilon$ to $\cup_j P_j$ is nonpositive with total mass $-1$, whence
    \begin{equation}
        M_\epsilon = |\sigma^-_\epsilon| \geq \sigma_\epsilon^-\brac{\cup_{j=1}^m P_j} =1.
    \end{equation}
    Therefore, we have both bounds and conclude.
\end{proof}
\subsection{Proof that the exponent $\frac{2}{d + 2}$ is sharp}\label{sec:sharp_proof} Given our construction $(\mu_\epsilon,\nu_\epsilon)$, it remains to establish the two bounds in \eqref{eq: contrad scale}; we split the proof into the lower and upper bounds. For the sake of clarity, given the sets $\cU$ and $\cZ$, we write $\la_{(u_i,z_i,r)} = \la_i$ and $\la^{(\rho)}_{(u_i,z_i,r)} = \la^{(\rho)}_{i}$. Recall that $\cU$ is an $\epsilon$-packing of $\R\mathbb{P}^{d-1}$ from Lemma \ref{lem: packing} and $\cZ$ is chosen according to Lemma \ref{lem: estimate}. Both measures are supported in $\cE^{(\rho)} \subseteq B_1(0)$.

We begin with the lower bound. By homogeneity of $W_1$ in Definition \ref{def: 1 homg} and Lemma \ref{lem: norm const}, it suffices to prove a lower bound of order $\epsilon^{2-2/d}$ on $W_1(\sigma^+_\epsilon, \sigma^-_\epsilon)$, since
$$W_1(\mu_\epsilon, \nu_\epsilon) = M_\epsilon^{-1} W_1(\sigma_\epsilon^+, \sigma_\epsilon^-)$$
and the normalizing constant adds only a factor $C_g$ independent of $\epsilon$. The function $\Phi(x) := d(x, \bigcup_{j=1}^m P_j)$ is $1$-Lipschitz and vanishes on $\bigcup_{j=1}^m P_j$. Therefore, by Kantorovich duality,
\begin{align*}
W_1(\sigma^+_\epsilon,\sigma^-_\epsilon) \geq \int_{\R^d} \Phi \,d(\sigma^+_\epsilon - \sigma^-_\epsilon)
    &= \frac{1}{m}\sum_{i=1}^m\int_{\R^d} \int_{-1}^1  g(s)\Phi(x + s\rho u_i)\,ds\,d\la_i \\
    &= c_g \rho + \frac{1}{m}\sum_{i=1}^m\int_{\R^d} \int_{-1}^1  g(s)(\Phi(x + s\rho u_i) - \rho|s|)\,ds\,d\la_i \\
    &\geq c_g \rho - C_{g,d}m\rho^2 r^{d-1}\,,
\end{align*}
where $c_g = \int_{-1}^1|s|g(s)\,ds$. Here the first equality uses that $\Phi$ vanishes on the support of each $\la_i$, together with the change of variables $x \mapsto x + s\rho u_i$, and the final inequality follows from the bound in Lemma \ref{lem: estimate}. By our choice of $m$ and $r$, the exponent of $\epsilon$ in $m\rho r^{d-1}$ vanishes:
$$
    m \rho r^{d-1} \leq c^{d-1} a^{d-1}\, a\, \epsilon^{2 - \frac{2}{d}-(d-1) + \frac{(d-1)(d-2)}{d}} = c^{d-1} a^{d}\,,
$$
and choosing $a$ small enough that $C_{g,d}\, c^{d-1} a^d < \frac{c_g}{2}$ gives $C_{g,d} m \rho^2 r^{d-1} < \frac{c_g}2 \rho$. Therefore, $W_1(\sigma_\epsilon^+, \sigma_\epsilon^-) \geq \frac{c_g}{2}\rho = \frac{ac_g}{2}\epsilon^{2-2/d}$, which proves the lower bound in \eqref{eq: contrad scale} with $\kappa_d = \frac{a c_g}{2 C_g}$.

Next, we complete the upper bound. As before, by homogeneity and Lemma \ref{lem: norm const} (using $M_\epsilon \geq 1$), it suffices to bound $\sup_{\theta \in \mathbb{S}^{d-1}}W_1((\theta_{\#}\sigma_\epsilon)^+, (\theta_\#\sigma_\epsilon)^-) = \sup_{\theta \in \mathbb{S}^{d-1}}\norm{\theta_\#\sigma_\epsilon}_{\text{KR}}$. Fix $\theta \in \mathbb{S}^{d-1}$ and let $\beta_i = \delta([u_i], [\theta]) = (1 - \abrac{u_i, \theta}^2)^{1/2}$. By the triangle inequality applied to $\norm{\cdot}_{\text{KR}}$ combined with Lemma \ref{lem: upper w_1},
\begin{align*}
    \norm{\theta_\#\sigma_\epsilon}_{\text{KR}} \leq \frac{1}{m}\sum_{i=1}^m \norm{\theta_\#(\la^{(\rho)}_i - \la_i)}_{\text{KR}}
    &{\leq} \frac{C_{d,g}}{m}\sum_{i=1}^m  r\min\sbrac{\epsilon, \frac{\epsilon^{d+2}}{\beta_i^{d+1}}}\,.
\end{align*}
By \eqref{eq: local entropy}, for every $\ell \geq 0$ there are at most $(\tilde c\, 2^{\ell+1})^{d-1}$ indices with $\beta_i \leq 2^{\ell} \epsilon$.
We decompose the sum accordingly: the indices with $\beta_i \leq \epsilon$ number at most $(2\tilde c)^{d-1}$ and contribute at most $r\epsilon$ each, while each index in the $\ell$th dyadic annulus $\sbrac{i : 2^{\ell}\epsilon < \beta_i \leq 2^{\ell+1}\epsilon}$ contributes at most $r\epsilon^{d+2}(2^{\ell}\epsilon)^{-(d+1)} = 2^{-\ell(d+1)}\, r\epsilon$. Therefore
\begin{align*}
	\sum_{i=1}^m r\min\sbrac{\epsilon, \frac{\epsilon^{d+2}}{\beta_i^{d+1}}} &\leq (2\tilde c)^{d-1} r\epsilon + \sum_{\ell = 0}^\infty (\tilde c\, 2^{\ell+1})^{d-1}\, 2^{-\ell(d+1)}\, r\epsilon \\
	&= (2\tilde c)^{d-1}\Big(1 + \sum_{\ell=0}^\infty 2^{-2\ell}\Big) r\epsilon \lesssim_d r\epsilon\,,
\end{align*}
and we obtain
\begin{equation}
	\norm{\theta_\#\sigma_\epsilon}_{\text{KR}} \lesssim_{d, g} \frac{r \epsilon}{m} \asymp \epsilon^{(d+1) - 2/d}\,.
\end{equation}
Since the bound is uniform in $\theta$, this proves the upper bound in \eqref{eq: contrad scale} and hence completes the proof of Theorem~\ref{thm: sharp exp}. \qed

\section{Comparison bounds under Condition~\ref{cond:voronoi}}\label{sec:discrete}

In this section we prove Theorems \ref{thm:discrete_bound} and \ref{thm:kdim_bound}. Throughout, $\mu \in \cP_p(\R^d)$ and $\nu$ is a probability measure supported on a countable closed set $\cY \subseteq \R^d$.
Since both sides of \eqref{eq:discrete_bound} and \eqref{eq:kdim_bound} vanish when $\mu = \nu$, and Condition~\ref{cond:voronoi} forces $\mu = \nu$ whenever $\mu(\cY) = 1$, we may assume $\mu(\R^d \setminus \cY) > 0$.
Since Theorem~\ref{thm:discrete_bound} is a special case of Theorem~\ref{thm:kdim_bound}, we proceed directly to the $k$-dimensional case.

The proof of Theorem~\ref{thm:kdim_bound} is based on the simple observation that Condition~\ref{cond:voronoi} gives rise to a closed-form expression for $W_p^p(\mu, \nu)$ and a compatible lower bound on $W_p^p(U_\#\mu, U_\#\nu)$ for any $U \in \mathfrak{G}_{d,k}$.

\begin{lem}\label{lem:voronoi_reduction}
	If Condition~\ref{cond:voronoi} holds, then for every $p \geq 1$,
	\begin{equation}
		W_p^p(\mu, \nu) = \int_{\R^d} d(x, \cY)^p \dd\mu(x)\,.
	\end{equation}
\end{lem}
\begin{proof}
	Every coupling $\gamma' \in \Gamma(\mu, \nu)$ is supported on $\R^d \times \cY$, so
	$$
		\int \norm{x-y}^p \dd\gamma'(x,y) \geq \int d(x, \cY)^p \dd\gamma'(x,y) = \int d(x,\cY)^p \dd\mu(x)\,,
	$$
	and hence $W_p^p(\mu, \nu) \geq \int d(x, \cY)^p \dd\mu$.
	Conversely, the coupling $\gamma$ furnished by Condition~\ref{cond:voronoi} achieves this bound, by~\eqref{eq:voronoi}.
\end{proof}

\begin{lem}\label{lem:projected_lb}
	For $1 \leq k < d$ and $U \in \mathfrak{G}_{d,k}$, write $U\cY := \sbrac{Uy : y \in \cY} \subseteq \R^k$ and $d(z, U\cY) := \inf_{y \in \cY}\norm{z - Uy}$. Then
	\begin{equation}
		W_p^p(U_\#\mu, U_\#\nu) \geq \int_{\R^d} d(Ux, U\cY)^p \dd\mu(x)\,.
	\end{equation}
	In particular, for $\theta \in \mathbb{S}^{d-1}$, writing $\cY_\theta := \sbrac{\abrac{\theta, y}: y \in \cY}$,
	\begin{equation}
		W_p^p(\theta_\#\mu, \theta_\#\nu) \geq \int_{\R^d} d(\abrac{\theta, x}, \cY_\theta)^p \dd\mu(x)\,.
	\end{equation}
\end{lem}
\begin{proof}
	The measure $U_\#\nu$ is supported on $U\cY$, so every coupling of $(U_\#\mu, U_\#\nu)$ is supported on $\R^k \times U\cY$ and therefore has cost at least $\int d(z, U\cY)^p \dd (U_\#\mu)(z) = \int d(Ux, U\cY)^p \dd\mu(x)$.
\end{proof}

We now proceed to the proof of the main theorem.

\begin{proof}[Proof of Theorem~\ref{thm:kdim_bound}]
	Abbreviate $K := K^{(k)}_{\mu, \nu}$; we may assume $K < \infty$.
	By Lemma~\ref{lem:projected_lb} and Fubini--Tonelli,
	\begin{equation}
		\SW_{p,k}^p(\mu, \nu) = \int_{\mathfrak{G}_{d,k}} W_p^p(U_\#\mu, U_\#\nu) \dd\sigma_{d,k}(U) \geq \int_{\R^d} \Esub{U \sim \sigma_{d,k}}{d(Ux, U\cY)^p} \dd\mu(x)\,.
	\end{equation}

	It therefore suffices to prove a suitable lower bound on $\Esub{U}{d(Ux, U\cY)^p}$.
	We employ an elementary anticoncentration bound on the Stiefel manifold, whose proof is deferred.
	
	\begin{lem}[Anti-concentration for $\sigma_{d,k}$]\label{lem: anti conc}
		Let $1 \leq k < d$ and let $U \sim \sigma_{d,k}$. There exists a universal constant $C_0 > 0$ such that for any $v\in \R^d$ and $t > 0$,
		\begin{equation}
			\Psub{U \sim \sigma_{d,k}}{\norm{Uv} < t\sqrt{k/d}\,\norm{v}} \leq (C_0 t)^{k}.
		\end{equation}
	\end{lem}
	
	For $x \notin \cY$, write $S_k(x) := \big(\sum_{y \in \cY} \norm{x-y}^{-k}\big)^{1/k}$.
	For any $\tau \geq 0$, applying Lemma~\ref{lem: anti conc} and the union bound yields
	\begin{equation*}
		\Prob{d(Ux, U\cY) \leq \tau} \leq \sum_{y \in \cY} \Prob{\norm{U(x-y)} \leq \tau} \leq \sum_{y \in \cY} \brac{\frac{C_0 \tau \sqrt{d/k}}{\norm{x-y}}}^k\,.
	\end{equation*}
	Choosing $\tau = \brac{2^{1/k}\, C_0 \sqrt{d/k}\, S_k(x)}^{-1}$ makes this probability at most $\frac 12$, so 
	\begin{equation}\label{eq:kdim_pointwise}
		\Esub{U \sim \sigma_{d,k}}{d(Ux, U\cY)^p} \geq \tau^p\, \Prob{d(Ux, U\cY) > \tau} \geq \frac{\tau^p}{2} = \frac 12 \brac{2^{1/k}\, C_0 \sqrt{\frac dk}\, S_k(x)}^{-p}\,.
	\end{equation}

	By the definition~\eqref{eq:kk_def} of $K^{(k)}$, for $\mu$-almost every $x \notin \cY$, we have $S_k(x)^{-1} \geq {d(x, \cY)/K}$, so \eqref{eq:kdim_pointwise} yields, for $\mu$-almost every $x$,
	\begin{equation*}
		\Esub{U \sim \sigma_{d,k}}{d(Ux, U\cY)^p} \geq \frac 12 \brac{2^{1/k}\, C_0 \sqrt{\frac dk}}^{-p} K^{-p}\, d(x, \cY)^p
	\end{equation*}
	(for $x \in \cY$ this holds trivially, since the right-hand side vanishes).
	Integrating against $\mu$ and applying Lemma~\ref{lem:voronoi_reduction},
	\begin{equation*}
		\SW_{p,k}^p(\mu, \nu) \geq \frac 12 \brac{2^{1/k}\, C_0 \sqrt{\frac dk}\, K}^{-p} W_p^p(\mu, \nu)\,.
	\end{equation*}
	Raising both sides to the power $1/p$ and using $2^{1/p} \leq 2$ and $2^{1/k} \leq 2$ gives
	\begin{equation*}
		W_p(\mu, \nu) \leq 4 C_0 \sqrt{\frac dk}\, K\, \SW_{p,k}(\mu, \nu)\,,
	\end{equation*}
	which proves the theorem with $C = 4C_0$.
\end{proof}

It remains to prove Lemma~\ref{lem: anti conc}.

\begin{proof}[Proof of Lemma~\ref{lem: anti conc}]
    It suffices to check the case $v = e_1$. By rotational invariance of the measure $\sigma_{d,k}$, $\norm{Ue_1}^2 = \sum_{i=1}^k \theta_{i}^2$ in distribution, where $\theta \sim \sigma_{d}$, and it is well known that $\sum_{i=1}^k \theta_i^2$ has a $\text{Beta}(k/2, (d-k)/2)$ distribution \cite{frankl1990some}.
    If $t \geq \frac 12$, the trivial bound $\Prob{\norm{Ue_1} < t\sqrt{k/d}} \leq 1 \leq (2t)^k$ suffices, so assume $0 < t < \frac 12$ and set $\eta := (k/d)t^2 \leq \frac{1}{4}$. Since $d - k \geq 1$, we have $(1-u)^{\frac{d-k}{2}-1} \leq (3/4)^{-1/2} \leq 2$ for every $u \in [0,\eta]$, and therefore
    \begin{align*}
        \Psub{U\sim\sigma_{d,k}}{\norm{Ue_1} < \eta^{1/2}} &\leq \frac{\Gamma(d/2)}{\Gamma(k/2)\Gamma((d-k)/2)}\int_0^\eta u^{\frac{k}{2}-1}(1-u)^{\frac{d-k}{2}-1}\,du \\
        &\leq \frac{4}{k} \cdot \frac{\Gamma(d/2)}{\Gamma(k/2)\Gamma((d-k)/2)}\, \eta^{k/2}\,.
    \end{align*}
    Since the digamma function $\psi = (\log \Gamma)'$ satisfies $\psi(u) \leq \log u$, we have $\log \Gamma(\frac d2) - \log\Gamma(\frac{d-k}2) = \int_{(d-k)/2}^{d/2}\psi(u)\dd u \leq \frac k2 \log \frac d2$, i.e., $\Gamma(d/2) \leq (d/2)^{k/2}\, \Gamma((d-k)/2)$. Hence, using $\eta^{k/2} = (k/d)^{k/2} t^k$,
    \begin{equation*}
        \Psub{U\sim\sigma_{d,k}}{\norm{Ue_1} < t\sqrt{k/d}} \leq \frac{4}{k}\cdot \frac{(k/2)^{k/2}}{\Gamma(k/2)}\, t^k \leq \frac{4}{k}\cdot \frac{(k/2)^{1/2}e^{k/2}}{\sqrt{2\pi}}\, t^k \leq (C_0 t)^{k}
    \end{equation*}
    for a universal constant $C_0$, where we used the Stirling bound $\Gamma(z) \geq \sqrt{2\pi}\, z^{z - 1/2} e^{-z}$ with $z = k/2$.
\end{proof}

\section{Lower bound: proof of Theorem~\ref{thm:discrete_lb}}\label{sec:lb}

In this section we show that the linear dependence on the complexity parameter $K_{\mu,\nu}$ in Theorem~\ref{thm:discrete_bound} cannot be improved, up to polylogarithmic factors. The example is due to Chen and Travaglini \cite{chen20101}.

Fix an integer $M > 1$ and let
\begin{equation}\label{eq:lattice_def}
	h := \frac{2}{2M+1}, \qquad N := (2M+1)^d, \qquad \Lambda_M := \sbrac{hj : j \in \sbrac{-M, \dots, M}^d} \subseteq Q := [-1,1]^d\,.
\end{equation}
The set $\Lambda_M$ consists of the centers of the $N$ subcubes of side length $h$ tiling $Q$. Let $\mu := \Unif(Q)$ be the normalized Lebesgue measure on $Q$ and let $\nu := \frac 1N \sum_{y \in \Lambda_M} \delta_y$.

For a unit vector $\theta \in \mathbb{S}^{d-1}$ and $t \in \R$, write $P_{\theta, t} := Q \cap \sbrac{x \in \R^d : \abrac{x, \theta} \leq t}$ for the intersection of $Q$ with a half-space. Chen and Travaglini prove \cite[Theorems 1 and 2]{chen20101} that the lattice $\Lambda_M$ satisfies
\begin{equation}\label{eq:chen_travaglini}
	\sup_{t \geq 0} \int_{\mathbb{S}^{d-1}} \abs{\operatorname{card}(\Lambda_M \cap P_{\theta,t}) - N 2^{-d} |P_{\theta,t}|}\, \dd\sigma_d(\theta) \leq c_d (\log N)^d\,,
\end{equation}
where $|P_{\theta,t}|$ denotes the Lebesgue measure of $P_{\theta, t}$.
By the symmetry, the bound~\eqref{eq:chen_travaglini} holds for all $t \in \R$.

The relevant properties of this pair $(\mu, \nu)$ are collected in the following lemma.

\begin{lem}\label{lem:lattice_props}
	Let $d \geq 2$, and let $\mu, \nu$ be as above. Then:
	\begin{enumerate}[(i)]
		\item Condition~\ref{cond:voronoi} holds for the pair $(\mu, \nu)$;
		\item $W_1(\mu, \nu) \geq \frac{h}{8}$;
		\item $K_{\mu, \nu} \asymp_d M^{d-1}$;
		\item $\SW_{1,1}(\mu, \nu) \leq C_d \frac{(\log N)^d}{N}$.
	\end{enumerate}
\end{lem}
\begin{proof}
	(i) For $y \in \Lambda_M$, the Voronoi cell of $y$ relative to $\Lambda_M$ intersected with $Q$ is the subcube $Q_y := y + [-\frac h2, \frac h2]^d$.
	The subcubes $\sbrac{Q_y}_{y \in \Lambda_M}$ tile $Q$ and each satisfies $\mu(Q_y) = \frac{h^d}{2^d} = \frac 1N = \nu(\sbrac y)$. The map $T$ sending $x$ to the nearest point of $\Lambda_M$ therefore pushes $\mu$ forward to $\nu$, and the coupling $\gamma := (\mathrm{id}, T)_\#\mu$ satisfies $\norm{x - T(x)} = d(x, \Lambda_M)$ for $\mu$-a.e.\ $x$, which is Condition~\ref{cond:voronoi}.

	(ii) By (i) and Lemma~\ref{lem:voronoi_reduction}, $W_1(\mu, \nu) = \int_Q d(x, \Lambda_M) \dd\mu(x)$. Since
	$$
		\mu\brac{x : d(x, \Lambda_M) < \tfrac h4} \leq N \cdot \frac{\omega_d (h/4)^d}{2^d} = \frac{\omega_d}{4^d} \leq \frac 12\,,
	$$
	where $\omega_d = \pi^{d/2}/\Gamma(\frac d2 + 1)$ is the volume of the unit ball, we get $W_1(\mu, \nu) \geq \frac h4 \cdot \frac 12 = \frac h8$.

	(iii) For the upper bound, fix $x \in Q$ and split the sum defining $K_{\mu,\nu}$ according to the distance from $x$. Since $d(x, \Lambda_M) \leq \frac{\sqrt d\, h}{2}$, and each term satisfies $d(x, \Lambda_M)/\norm{x - y} \leq 1$,
	\begin{align*}
		d(x, \Lambda_M) \sum_{y \in \Lambda_M} \frac{1}{\norm{x-y}} &\leq \operatorname{card}\sbrac{y : \norm{x - y} \leq 2\sqrt d h}\\
		&\qquad+ \frac{\sqrt d h}{2}\sum_{\substack{y \in \Lambda_M \\ \norm{x-y} > 2\sqrt d h}} \frac{1}{\norm{x - y}}\,.
	\end{align*}
	The first term is at most $C_d$. For the second, grouping the lattice points into shells $\sbrac{y : mh \leq \norm{x - y} < (m+1)h}$, each containing at most $C_d m^{d-1}$ points, with $m$ ranging up to $O(\sqrt d/h) = O_d(M)$,
	$$
		\sum_{\substack{y \in \Lambda_M \\ \norm{x-y} > 2\sqrt d h}} \frac{1}{\norm{x-y}} \leq \frac{C_d}{h}\sum_{m = 1}^{O_d(M)} m^{d-2} \leq \frac{C_d'}{h}\, M^{d-1}\,,
	$$
	so that $d(x, \Lambda_M) \sum_y \norm{x-y}^{-1} \leq C_d M^{d-1}$, and hence $K_{\mu,\nu} \lesssim_d M^{d-1}$.
	For the lower bound, note first that $K_{\mu,\nu} \geq 1$, as observed after \eqref{eq:k_def}; since $M^{d-1} \lesssim_d 1$ for $M \leq 8$, this proves the claim when $M \leq 8$, so assume $M \geq 8$, i.e., $h \leq \frac 18$.
	Consider the set $A$ of points $x$ with $\norm{x - x_*} \leq \frac h8$, where $x_* := \frac h2 e_1$; then $\mu(A) > 0$ and every $x \in A$ satisfies $d(x, \Lambda_M) \geq \frac h2 - \frac h8 = \frac{3h}8$.
	Let $\cY_0 := \Lambda_M \cap [\tfrac 14, \tfrac 12]^d$. Since $h \leq \frac 18$, each coordinate interval $[\tfrac 14, \tfrac 12]$ contains at least $\frac{1}{4h} - 1 \geq \frac{1}{8h}$ points of $h\Z$, all of the form $hj$ with $0 \leq j \leq M$, so $\operatorname{card}(\cY_0) \geq (8h)^{-d}$.
	Moreover, for $x \in A$ and $y \in \cY_0$, we have $\norm{x - y} \leq \norm{x} + \norm{y} \leq h + \frac{\sqrt d}{2} \leq \sqrt d$. Therefore, for every $x \in A$,
	$$
		d(x, \Lambda_M) \sum_{y \in \Lambda_M} \frac{1}{\norm{x - y}} \geq \frac{3h}{8} \cdot \frac{\operatorname{card}(\cY_0)}{\sqrt d} \geq \frac{3}{8^{d+1}\sqrt{d}}\, h^{1-d} \gtrsim_d M^{d-1}\,,
	$$
	and hence $K_{\mu, \nu} \gtrsim_d M^{d-1}$.

	(iv) For one-dimensional probability measures, $W_1$ coincides with the $L^1$ distance between cumulative distribution functions \cite{bobkov2019one}. For fixed $\theta$, the distribution functions of $\theta_\#\mu$ and $\theta_\#\nu$ at $t \in \R$ are $2^{-d}|P_{\theta,t}|$ and $N^{-1}\operatorname{card}(\Lambda_M \cap P_{\theta,t})$ respectively, and both distributions are supported in $[-\sqrt d, \sqrt d]$. Therefore, by Fubini--Tonelli and \eqref{eq:chen_travaglini} (extended to all $t \in \R$ as noted above),
	\begin{align*}
		\SW_{1,1}(\mu,\nu) &= \int_{\mathbb{S}^{d-1}} \int_{-\sqrt d}^{\sqrt d} \abs{\frac{\operatorname{card}(\Lambda_M \cap P_{\theta,t})}{N} - \frac{|P_{\theta,t}|}{2^d}}\dd t \dd \sigma_d(\theta) \\
		&\leq \frac{2\sqrt d}{N} \sup_{t \in \R}\int_{\mathbb{S}^{d-1}} \abs{\operatorname{card}(\Lambda_M \cap P_{\theta,t}) - N2^{-d}|P_{\theta,t}|}\dd\sigma_d(\theta) \leq C_d \frac{(\log N)^d}{N}\,. \qedhere
	\end{align*}
\end{proof}

\begin{proof}[Proof of Theorem~\ref{thm:discrete_lb}]
	Given $K \geq 1$, set $M := \max\brac{\ceil{K^{1/(d-1)}}, 2}$ and let $(\mu, \nu)$ be the pair constructed above. By Lemma~\ref{lem:lattice_props}(iii), $K_{\mu, \nu} \asymp_d M^{d-1} \asymp_d K$. Combining parts (ii) and (iv) of Lemma~\ref{lem:lattice_props}, and using $Nh = 2(2M+1)^{d-1} \asymp M^{d-1}$ and $\log N = d \log(2M+1) \asymp_d \log(K+1)$,
	\begin{equation}
		\frac{W_1(\mu,\nu)}{\SW_{1,1}(\mu,\nu)} \geq \frac{h/8}{C_d (\log N)^d/N} = \frac{Nh}{8C_d(\log N)^d} \geq c_d \frac{K}{(\log (K+1))^{d}}\,,
	\end{equation}
	which is the claimed bound.
\end{proof}

\section*{AI Disclosure}
Claude Fable 5 and ChatGPT 5.6 were used in editing the manuscript, checking for errors, and filling in routine technical details; in particular, these models proved useful in correcting an error in a previous version of Theorem~\ref{thm:discrete_bound}. The authors developed the arguments and wrote the manuscript.

\section*{Acknowledgments}
The authors are grateful to Yanjun Han, N\'estor Guill\'en, Jaume de Dios Pont, Shay Sadovsky, and Uriel Mart\'inez Le\'on for helpful discussions.
J.N.W. was supported in part by NSF grant DMS-2339829.
J.S. was supported by the NYU MacCracken Fellowship.

\printbibliography
\end{document}

%% file: shortcuts.tex
    \DeclarePairedDelimiter\ceil{\lceil}{
\rceil}
    \DeclarePairedDelimiter\floor{\lfloor}{\rfloor}

    \newtheorem{thm}{Theorem}[section]

    \newtheorem{lem}[thm]{Lemma}

    \theoremstyle{remark}
    
    \theoremstyle{definition}
    \newtheorem{definition}[thm]{Definition}
    
    \numberwithin{equation}{section}

    \newcommand{\bean}{\begin{eqnarray}}
    \newcommand{\eean}{\end{eqnarray}}
    \newcommand{\be}{\begin{displaymath}}
    \newcommand{\ee}{\end{displaymath}}
    \newcommand{\bea}{\begin{eqnarray*}}   
    \newcommand{\eea}{\end{eqnarray*}}

    \newcommand{\nc}{\newcommand}
    
    \makeatletter
    \newcommand{\oset}[3][0ex]{%
      \mathrel{\mathop{#3}\limits^{
        \vbox to#1{\kern-2\ex@
        \hbox{$\scriptstyle#2$}\vss}}}}
    \makeatother
    \newcommand{\brac}[1]{\left( #1 \right)}                                
    \newcommand{\sbrac}[1]{\left\{ #1 \right\}}                             
    \newcommand{\abrac}[1]{\left\langle #1\right\rangle}                    
    
    \newcommand{\Unif}{\operatorname{Unif}}
    
    \newcommand{\E}[1]{\mathbb{E}\left[#1\right]} 							
    
    \newcommand{\Prob}[1]{\mathbb{P}\left(#1\right)}						
    \newcommand{\Psub}[2]{{\mathbb P}_{#1}\left(#2\right)}		      	

    \newcommand{\Esub}[2]{{\mathbb E_{#1}}\left[#2\right]}
    
    \newcommand{\iid}{\text{i.i.d.}}
    
    \DeclareDocumentCommand{\Pwto} {o} {
    	\IfNoValueTF {#1}
    	{\overset{n\to\infty}{\longrightarrow}}
    	{ \xrightarrow[ #1 \to \infty]{\Pr }}
    }
    \DeclareDocumentCommand{\Prto} {o} {
    	\IfNoValueTF {#1}
    	{\overset{\Pr}{\longrightarrow}}
    	{ \xrightarrow[ #1 \to \infty]{\Pr }}
    }
    \DeclareDocumentCommand{\Asto} {o} {
    	\IfNoValueTF {#1}
    	{\overset{\operatorname{a.s.}}{\longrightarrow}}
    	{
    		\xrightarrow[ #1 \to \infty]{\operatorname{a.s.} }
    	}
    }
    \DeclareDocumentCommand{\Mgfto} {o} {
    	\IfNoValueTF {#1}
    	{\overset{\operatorname{mgf}}{\longrightarrow}}
    	{ \xrightarrow[ #1 \to \infty]{\operatorname{mgf} }}
    }
    
    \DeclareDocumentCommand{\Wkto} {o} {
    	\IfNoValueTF {#1}
    	{\overset{(d)}{\longrightarrow}}
    	{ \xrightarrow[ #1 \to \infty]{(d) }}
    }
    
    \DeclareDocumentCommand \LPto { O{1} }
    {\overset{\operatorname{L^{#1}}}{\longrightarrow}}
    \newcommand\restr[2]{{
      \left.\kern-\nulldelimiterspace 
      #1 
      \vphantom{\big|} 
      \right|_{#2} 
      }}
    
    \nc{\on}{\operatorname}
    \nc{\ch}{\mbox{ch}}
    \nc{\Z}{{\mathbb Z}}
    \nc{\C}{{\mathbb C}}
    \nc{\prob}{\mathbb{P}}
    
    \nc{\indic}{\mathbf{1}}
    \nc{\indistr}{\overset{d}{\longrightarrow}}
    \nc{\inprob}{\overset{p}{\longrightarrow}}
    \nc{\N}{{\mathbb N}}
    \nc{\pone}{{\mathbb C}{\mathbb P}^1}
    \nc{\pa}{\partial}
    \nc{\F}{{\mathcal F}}
    \nc{\arr}{\rightarrow}
    \nc{\larr}{\longrightarrow}
    \nc{\al}{\alpha}
    \nc{\ri}{\rangle}
    \nc{\lef}{\langle}
    \nc{\W}{{\mathbb W}}
    \nc{\la}{\lambda}
    \nc{\lapl}{\mathcal{L}}
    \nc{\ep}{\lambda}
    \nc{\su}{\widehat{{\mathfrak s}{\mathfrak l}}_2}
    \nc{\sw}{{\mathfrak s}{\mathfrak l}}
    \nc{\g}{{\mathfrak g}}
    \nc{\h}{{\mathfrak h}}
    \nc{\n}{{\mathfrak n}}
    \nc{\G}{\widehat{\g}}
    \nc{\De}{\Delta}
    \nc{\gt}{\widetilde{\g}}
    \nc{\Ga}{\Gamma}
    \nc{\one}{{\mathbf 1}}
    \nc{\z}{{\mathfrak Z}}
    \nc{\La}{\Lambda}
    \nc{\wt}{\widetilde}
    \nc{\wh}{\widehat}
    \nc{\cri}{_{\kappa_c}}
    \nc{\si}{\sigma}
    \nc{\el}{\ell}
    \nc{\bi}{\bibitem}
    \nc{\om}{\omega}
    \nc{\ol}{\overline}
    \nc{\dzz}{\frac{dz}{z}}
    \nc{\mc}{\mathcal}
    \nc{\Cal}{\mathcal}
    \nc{\bb}{{\mathfrak b}}
    \nc{\ot}{\otimes}
    \nc{\R}{{\mathbb R}}
    \nc{\yy}{{\mc Y}}
    \nc{\ga}{\gamma}
    \nc{\us}{\underset}
    \nc{\opl}{\oplus}
    \nc{\Fq}{{\mathbb F}_q}
    \nc{\Mq}{{\mathcal M}}
    \nc{\Rep}{\on{Rep}}
    \nc{\sssec}{\subsubsection}
    \nc{\ssec}{\subsection}
    \nc{\lan}{\langle}
    \nc{\ran}{\rangle}

    \newcommand\cD{\mathcal D}
    \newcommand\cE{\mathcal E}

    \newcommand\cH{\mathcal H}

    \newcommand\cP{\mathcal P}

    \newcommand\cU{{\mathcal U}}

    \newcommand\cY{{\mathcal Y}}
    \newcommand\cZ{{\mathcal Z}}
    
    \nc{\D}{\mathcal D}
    \nc{\Vect}{\on{Vect}}
    \nc{\ghat}{\G}
    \nc{\T}{\mc T}
    \nc{\Tloc}{\T^\g_{\on{loc}}}
    \nc{\vac}{|0\ran}
    \nc{\Wick}{{\mb :}}
    \nc{\mb}{\mathbf}
    \nc{\delz}{\partial_z}
    \nc{\K}{{\cali K}}
    \nc{\cali}{\mathcal}
    \nc{\li}{\mathfrak l}
    \nc{\lt}{\widetilde{\li}}
    \nc{\astar}{a^*}
    \nc{\ka}{\kappa}
    
    \nc{\OO}{{\mc O}}
    \nc{\AutO}{\on{Aut}\OO}
    \nc{\DerO}{\on{Der}\OO}
    \nc{\DerpO}{\on{Der}_+\OO}
    \nc{\Au}{{\mc A}ut}
    \nc{\mf}{\mathfrak}
    \nc{\V}{{\mc V}}
    \nc{\hh}{\wh{\h}}
    
    \nc{\pp}{{\mathfrak p}}
    \nc{\mm}{{\mathfrak m}}
    \nc{\rr}{{\mathfrak r}}
    \nc{\gr}{\on{gr}}
    \nc{\Spe}{\on{Spec}}
    \nc{\rv}{\rho^\vee}
    \nc{\can}{\on{can}}
    \nc{\CC}{\on{Op}_G(D))}
    \nc{\Op}{\on{Op}_G(D)}
    \nc{\MOp}{\on{MOp}_G(D)}
    \nc{\Db}{{\mathbb D}}
    \nc{\ww}{w}
    
    \nc{\oQl}{\ol{{\mathbb Q}}_\ell}
    \nc{\oFq}{\ol{{\mathbb F}}_q}
    \nc{\Q}{{\mathbb Q}}
    \nc{\Ql}{{\mathbb Q}_\ell}
    \nc{\bs}{\backslash}
    \nc{\AD}{{\mathbb A}}
    \nc{\M}{{\mc M}}
    \nc{\Bun}{\on{Bun}}
    \nc{\hl}{h^{\leftarrow}}
    \nc{\hr}{h^{\rightarrow}}
    \nc{\supp}{\on{supp}}
    \nc{\He}{\on{H}}
    \nc{\Aut}{\on{Aut}}
    \nc{\Ll}{{\mc L}}
    \nc{\Coh}{{{\mathcal C}oh}}
    \nc{\ovc}{\overset{\circ}}
    \nc{\Hav}{\on{H}}
    \nc{\Mod}{\on{Mod}}
    \nc{\kk}{{\mathfrak k}}
    \nc{\vf}{\varphi} 
    \nc{\Gr}{\on{Gr}}
    \nc{\gen}{\on{gen}}
    \nc{\IC}{\on{IC}}
    \nc{\Jac}{\on{Jac}}

%% file: figures/sliced_diagram.tex
\begin{tikzpicture}
 
\begin{scope}[shift={(0,0)}]
  \draw[gray!70] (-2.8,0) -- (2.8,0);
  \draw[gray!70] (0,-2.8) -- (0,2.8);
 
  \foreach \a in {45,90,135,180}{
    \filldraw[rotate=\a, fill=gray!35, fill opacity=0.5,
              draw=gray!70, very thin]
      (0,0) ellipse (2 and 0.22);
  }
 
  \draw[very thin, gray!75] (0,0) -- (45:2);
  \draw[very thin, gray!75] (0,0) -- (135:0.22);
  \node[inner sep=1pt] at ($(45:1)+(-45:0.34)$) {\scriptsize $1$};
  \node[inner sep=1pt] at ($(135:0.11)+(225:0.24)$) {\scriptsize $\rho$};
 
  \fill[gray!70] (0,0) circle (1.2pt);
 
  \draw[->, very thin] (0,0) -- (55:2)
      node[above right, inner sep=1pt] {$\theta$};
 
  \node[above] at (0,2.8) {$\mu$};
\end{scope}
 
\begin{scope}[shift={(8.2,0)}]
  \draw[gray!70] (-2.8,0) -- (2.8,0);
  \draw[gray!70] (0,-2.8) -- (0,2.8);
 
  \foreach \a in {45,90,135,180}{
    \filldraw[rotate=\a, fill=gray!35, fill opacity=0.5,
              draw=gray!70, very thin]
      (0,0.3) ellipse (2 and 0.22);
    \filldraw[rotate=\a, fill=gray!35, fill opacity=0.5,
              draw=gray!70, very thin]
      (0,-0.3) ellipse (2 and 0.22);
  }
 
  \draw[very thin, gray!75] (135:0.3) -- ($(135:0.3)+(45:2)$);
  \draw[very thin, gray!75] (135:0.3) -- ($(135:0.3)+(135:0.22)$);
 
  \foreach \a in {45,90,135,180}{
    \fill[gray!70, rotate=\a] (0,0.3)  circle (1.2pt);
    \fill[gray!70, rotate=\a] (0,-0.3) circle (1.2pt);
  }
 
  \draw[->, very thin] (0,0) -- (55:2)
      node[above right, inner sep=1pt] {$\theta$};
 
  \node[above] at (0,2.8) {$\nu$};
\end{scope}
 
\begin{scope}[shift={(4.1,-6.4)}, scale=1.2]
  \draw[gray!70] (-3.5,0) -- (3.5,0);
  \draw[gray!70] (0,-0.15) -- (0,2.2);
 
  \fill[gray!35, fill opacity=0.5]
    (-3.400,0) -- plot coordinates {(-3.400,0.043) (-3.389,0.043) (-3.379,0.044) (-3.368,0.044) (-3.357,0.045) (-3.347,0.045) (-3.336,0.046) (-3.326,0.046) (-3.315,0.047) (-3.304,0.047) (-3.294,0.048) (-3.283,0.048) (-3.272,0.049) (-3.262,0.049) (-3.251,0.050) (-3.241,0.050) (-3.230,0.051) (-3.219,0.052) (-3.209,0.052) (-3.198,0.053) (-3.188,0.053) (-3.177,0.054) (-3.166,0.055) (-3.156,0.055) (-3.145,0.056) (-3.134,0.056) (-3.124,0.057) (-3.113,0.058) (-3.103,0.058) (-3.092,0.059) (-3.081,0.060) (-3.071,0.060) (-3.060,0.061) (-3.049,0.061) (-3.039,0.062) (-3.028,0.063) (-3.018,0.063) (-3.007,0.064) (-2.996,0.065) (-2.986,0.066) (-2.975,0.066) (-2.964,0.067) (-2.954,0.068) (-2.943,0.068) (-2.933,0.069) (-2.922,0.070) (-2.911,0.070) (-2.901,0.071) (-2.890,0.072) (-2.879,0.073) (-2.869,0.073) (-2.858,0.074) (-2.848,0.075) (-2.837,0.076) (-2.826,0.077) (-2.816,0.077) (-2.805,0.078) (-2.794,0.079) (-2.784,0.080) (-2.773,0.081) (-2.763,0.081) (-2.752,0.082) (-2.741,0.083) (-2.731,0.084) (-2.720,0.085) (-2.709,0.085) (-2.699,0.086) (-2.688,0.087) (-2.678,0.088) (-2.667,0.089) (-2.656,0.090) (-2.646,0.091) (-2.635,0.091) (-2.624,0.092) (-2.614,0.093) (-2.603,0.094) (-2.593,0.095) (-2.582,0.096) (-2.571,0.097) (-2.561,0.098) (-2.550,0.099) (-2.539,0.100) (-2.529,0.101) (-2.518,0.102) (-2.508,0.102) (-2.497,0.103) (-2.486,0.104) (-2.476,0.105) (-2.465,0.106) (-2.454,0.107) (-2.444,0.108) (-2.433,0.109) (-2.423,0.110) (-2.412,0.111) (-2.401,0.112) (-2.391,0.113) (-2.380,0.114) (-2.369,0.115) (-2.359,0.116) (-2.348,0.117) (-2.338,0.118) (-2.327,0.119) (-2.316,0.121) (-2.306,0.122) (-2.295,0.123) (-2.284,0.124) (-2.274,0.125) (-2.263,0.126) (-2.252,0.127) (-2.242,0.128) (-2.231,0.129) (-2.221,0.130) (-2.210,0.131) (-2.199,0.132) (-2.189,0.133) (-2.178,0.135) (-2.167,0.136) (-2.157,0.137) (-2.146,0.138) (-2.136,0.139) (-2.125,0.140) (-2.114,0.141) (-2.104,0.142) (-2.093,0.144) (-2.083,0.145) (-2.072,0.146) (-2.061,0.147) (-2.051,0.148) (-2.040,0.149) (-2.029,0.151) (-2.019,0.152) (-2.008,0.153) (-1.998,0.154) (-1.987,0.155) (-1.976,0.156) (-1.966,0.158) (-1.955,0.159) (-1.944,0.160) (-1.934,0.161) (-1.923,0.162) (-1.913,0.164) (-1.902,0.165) (-1.891,0.166) (-1.881,0.167) (-1.870,0.168) (-1.859,0.170) (-1.849,0.171) (-1.838,0.172) (-1.828,0.173) (-1.817,0.175) (-1.806,0.176) (-1.796,0.177) (-1.785,0.178) (-1.774,0.180) (-1.764,0.181) (-1.753,0.182) (-1.743,0.183) (-1.732,0.184) (-1.721,0.186) (-1.711,0.187) (-1.700,0.188) (-1.689,0.189) (-1.679,0.191) (-1.668,0.192) (-1.657,0.193) (-1.647,0.194) (-1.636,0.196) (-1.626,0.197) (-1.615,0.198) (-1.604,0.199) (-1.594,0.201) (-1.583,0.202) (-1.573,0.203) (-1.562,0.204) (-1.551,0.206) (-1.541,0.207) (-1.530,0.208) (-1.519,0.209) (-1.509,0.211) (-1.498,0.212) (-1.488,0.213) (-1.477,0.214) (-1.466,0.216) (-1.456,0.217) (-1.445,0.218) (-1.434,0.219) (-1.424,0.221) (-1.413,0.222) (-1.403,0.223) (-1.392,0.224) (-1.381,0.226) (-1.371,0.227) (-1.360,0.228) (-1.349,0.229) (-1.339,0.231) (-1.328,0.232) (-1.318,0.233) (-1.307,0.234) (-1.296,0.235) (-1.286,0.237) (-1.275,0.238) (-1.264,0.239) (-1.254,0.240) (-1.243,0.241) (-1.233,0.243) (-1.222,0.244) (-1.211,0.245) (-1.201,0.246) (-1.190,0.247) (-1.179,0.248) (-1.169,0.250) (-1.158,0.251) (-1.147,0.252) (-1.137,0.253) (-1.126,0.254) (-1.116,0.255) (-1.105,0.257) (-1.094,0.258) (-1.084,0.259) (-1.073,0.260) (-1.062,0.261) (-1.052,0.262) (-1.041,0.263) (-1.031,0.264) (-1.020,0.265) (-1.009,0.266) (-0.999,0.267) (-0.988,0.269) (-0.977,0.270) (-0.967,0.271) (-0.956,0.272) (-0.946,0.273) (-0.935,0.274) (-0.924,0.275) (-0.914,0.276) (-0.903,0.277) (-0.893,0.278) (-0.882,0.279) (-0.871,0.280) (-0.861,0.281) (-0.850,0.282) (-0.839,0.283) (-0.829,0.284) (-0.818,0.285) (-0.808,0.285) (-0.797,0.286) (-0.786,0.287) (-0.776,0.288) (-0.765,0.289) (-0.754,0.290) (-0.744,0.291) (-0.733,0.292) (-0.723,0.293) (-0.712,0.293) (-0.701,0.294) (-0.691,0.295) (-0.680,0.296) (-0.669,0.297) (-0.659,0.298) (-0.648,0.298) (-0.638,0.299) (-0.627,0.300) (-0.616,0.301) (-0.606,0.302) (-0.595,0.303) (-0.584,0.305) (-0.574,0.306) (-0.563,0.308) (-0.552,0.310) (-0.542,0.312) (-0.531,0.316) (-0.521,0.319) (-0.510,0.324) (-0.499,0.330) (-0.489,0.338) (-0.478,0.347) (-0.467,0.358) (-0.457,0.371) (-0.446,0.386) (-0.436,0.405) (-0.425,0.427) (-0.414,0.452) (-0.404,0.481) (-0.393,0.514) (-0.383,0.550) (-0.372,0.590) (-0.361,0.633) (-0.351,0.680) (-0.340,0.729) (-0.329,0.779) (-0.319,0.831) (-0.308,0.883) (-0.298,0.933) (-0.287,0.982) (-0.276,1.027) (-0.266,1.067) (-0.255,1.102) (-0.244,1.130) (-0.234,1.151) (-0.223,1.164) (-0.213,1.168) (-0.202,1.165) (-0.191,1.152) (-0.181,1.132) (-0.170,1.105) (-0.159,1.072) (-0.149,1.032) (-0.138,0.989) (-0.128,0.943) (-0.117,0.895) (-0.106,0.846) (-0.096,0.798) (-0.085,0.752) (-0.074,0.709) (-0.064,0.670) (-0.053,0.635) (-0.043,0.606) (-0.032,0.583) (-0.021,0.566) (-0.011,0.556) (0.000,0.553) (0.011,0.556) (0.021,0.566) (0.032,0.583) (0.043,0.606) (0.053,0.635) (0.064,0.670) (0.074,0.709) (0.085,0.752) (0.096,0.798) (0.106,0.846) (0.117,0.895) (0.128,0.943) (0.138,0.989) (0.149,1.032) (0.159,1.072) (0.170,1.105) (0.181,1.132) (0.191,1.152) (0.202,1.165) (0.212,1.168) (0.223,1.164) (0.234,1.151) (0.244,1.130) (0.255,1.102) (0.266,1.067) (0.276,1.027) (0.287,0.982) (0.297,0.933) (0.308,0.883) (0.319,0.831) (0.329,0.779) (0.340,0.729) (0.351,0.680) (0.361,0.633) (0.372,0.590) (0.382,0.550) (0.393,0.514) (0.404,0.481) (0.414,0.452) (0.425,0.427) (0.436,0.405) (0.446,0.386) (0.457,0.371) (0.467,0.358) (0.478,0.347) (0.489,0.338) (0.499,0.330) (0.510,0.324) (0.521,0.319) (0.531,0.316) (0.542,0.312) (0.552,0.310) (0.563,0.308) (0.574,0.306) (0.584,0.305) (0.595,0.303) (0.606,0.302) (0.616,0.301) (0.627,0.300) (0.638,0.299) (0.648,0.298) (0.659,0.298) (0.669,0.297) (0.680,0.296) (0.691,0.295) (0.701,0.294) (0.712,0.293) (0.722,0.293) (0.733,0.292) (0.744,0.291) (0.754,0.290) (0.765,0.289) (0.776,0.288) (0.786,0.287) (0.797,0.286) (0.807,0.285) (0.818,0.285) (0.829,0.284) (0.839,0.283) (0.850,0.282) (0.861,0.281) (0.871,0.280) (0.882,0.279) (0.892,0.278) (0.903,0.277) (0.914,0.276) (0.924,0.275) (0.935,0.274) (0.946,0.273) (0.956,0.272) (0.967,0.271) (0.977,0.270) (0.988,0.269) (0.999,0.267) (1.009,0.266) (1.020,0.265) (1.031,0.264) (1.041,0.263) (1.052,0.262) (1.062,0.261) (1.073,0.260) (1.084,0.259) (1.094,0.258) (1.105,0.257) (1.116,0.255) (1.126,0.254) (1.137,0.253) (1.147,0.252) (1.158,0.251) (1.169,0.250) (1.179,0.248) (1.190,0.247) (1.201,0.246) (1.211,0.245) (1.222,0.244) (1.233,0.243) (1.243,0.241) (1.254,0.240) (1.264,0.239) (1.275,0.238) (1.286,0.237) (1.296,0.235) (1.307,0.234) (1.318,0.233) (1.328,0.232) (1.339,0.231) (1.349,0.229) (1.360,0.228) (1.371,0.227) (1.381,0.226) (1.392,0.224) (1.403,0.223) (1.413,0.222) (1.424,0.221) (1.434,0.219) (1.445,0.218) (1.456,0.217) (1.466,0.216) (1.477,0.214) (1.488,0.213) (1.498,0.212) (1.509,0.211) (1.519,0.209) (1.530,0.208) (1.541,0.207) (1.551,0.206) (1.562,0.204) (1.572,0.203) (1.583,0.202) (1.594,0.201) (1.604,0.199) (1.615,0.198) (1.626,0.197) (1.636,0.196) (1.647,0.194) (1.657,0.193) (1.668,0.192) (1.679,0.191) (1.689,0.189) (1.700,0.188) (1.711,0.187) (1.721,0.186) (1.732,0.184) (1.742,0.183) (1.753,0.182) (1.764,0.181) (1.774,0.180) (1.785,0.178) (1.796,0.177) (1.806,0.176) (1.817,0.175) (1.827,0.173) (1.838,0.172) (1.849,0.171) (1.859,0.170) (1.870,0.168) (1.881,0.167) (1.891,0.166) (1.902,0.165) (1.912,0.164) (1.923,0.162) (1.934,0.161) (1.944,0.160) (1.955,0.159) (1.966,0.158) (1.976,0.156) (1.987,0.155) (1.997,0.154) (2.008,0.153) (2.019,0.152) (2.029,0.151) (2.040,0.149) (2.051,0.148) (2.061,0.147) (2.072,0.146) (2.083,0.145) (2.093,0.144) (2.104,0.142) (2.114,0.141) (2.125,0.140) (2.136,0.139) (2.146,0.138) (2.157,0.137) (2.167,0.136) (2.178,0.135) (2.189,0.133) (2.199,0.132) (2.210,0.131) (2.221,0.130) (2.231,0.129) (2.242,0.128) (2.252,0.127) (2.263,0.126) (2.274,0.125) (2.284,0.124) (2.295,0.123) (2.306,0.122) (2.316,0.121) (2.327,0.119) (2.338,0.118) (2.348,0.117) (2.359,0.116) (2.369,0.115) (2.380,0.114) (2.391,0.113) (2.401,0.112) (2.412,0.111) (2.423,0.110) (2.433,0.109) (2.444,0.108) (2.454,0.107) (2.465,0.106) (2.476,0.105) (2.486,0.104) (2.497,0.103) (2.508,0.102) (2.518,0.102) (2.529,0.101) (2.539,0.100) (2.550,0.099) (2.561,0.098) (2.571,0.097) (2.582,0.096) (2.592,0.095) (2.603,0.094) (2.614,0.093) (2.624,0.092) (2.635,0.091) (2.646,0.091) (2.656,0.090) (2.667,0.089) (2.677,0.088) (2.688,0.087) (2.699,0.086) (2.709,0.085) (2.720,0.085) (2.731,0.084) (2.741,0.083) (2.752,0.082) (2.762,0.081) (2.773,0.081) (2.784,0.080) (2.794,0.079) (2.805,0.078) (2.816,0.077) (2.826,0.077) (2.837,0.076) (2.847,0.075) (2.858,0.074) (2.869,0.073) (2.879,0.073) (2.890,0.072) (2.901,0.071) (2.911,0.070) (2.922,0.070) (2.932,0.069) (2.943,0.068) (2.954,0.068) (2.964,0.067) (2.975,0.066) (2.986,0.066) (2.996,0.065) (3.007,0.064) (3.018,0.063) (3.028,0.063) (3.039,0.062) (3.049,0.061) (3.060,0.061) (3.071,0.060) (3.081,0.060) (3.092,0.059) (3.103,0.058) (3.113,0.058) (3.124,0.057) (3.134,0.056) (3.145,0.056) (3.156,0.055) (3.166,0.055) (3.177,0.054) (3.188,0.053) (3.198,0.053) (3.209,0.052) (3.219,0.052) (3.230,0.051) (3.241,0.050) (3.251,0.050) (3.262,0.049) (3.272,0.049) (3.283,0.048) (3.294,0.048) (3.304,0.047) (3.315,0.047) (3.326,0.046) (3.336,0.046) (3.347,0.045) (3.357,0.045) (3.368,0.044) (3.379,0.044) (3.389,0.043) (3.400,0.043)} -- (3.400,0) -- cycle;
  \draw[gray!90, thin] plot coordinates {(-3.400,0.043) (-3.389,0.043) (-3.379,0.044) (-3.368,0.044) (-3.357,0.045) (-3.347,0.045) (-3.336,0.046) (-3.326,0.046) (-3.315,0.047) (-3.304,0.047) (-3.294,0.048) (-3.283,0.048) (-3.272,0.049) (-3.262,0.049) (-3.251,0.050) (-3.241,0.050) (-3.230,0.051) (-3.219,0.052) (-3.209,0.052) (-3.198,0.053) (-3.188,0.053) (-3.177,0.054) (-3.166,0.055) (-3.156,0.055) (-3.145,0.056) (-3.134,0.056) (-3.124,0.057) (-3.113,0.058) (-3.103,0.058) (-3.092,0.059) (-3.081,0.060) (-3.071,0.060) (-3.060,0.061) (-3.049,0.061) (-3.039,0.062) (-3.028,0.063) (-3.018,0.063) (-3.007,0.064) (-2.996,0.065) (-2.986,0.066) (-2.975,0.066) (-2.964,0.067) (-2.954,0.068) (-2.943,0.068) (-2.933,0.069) (-2.922,0.070) (-2.911,0.070) (-2.901,0.071) (-2.890,0.072) (-2.879,0.073) (-2.869,0.073) (-2.858,0.074) (-2.848,0.075) (-2.837,0.076) (-2.826,0.077) (-2.816,0.077) (-2.805,0.078) (-2.794,0.079) (-2.784,0.080) (-2.773,0.081) (-2.763,0.081) (-2.752,0.082) (-2.741,0.083) (-2.731,0.084) (-2.720,0.085) (-2.709,0.085) (-2.699,0.086) (-2.688,0.087) (-2.678,0.088) (-2.667,0.089) (-2.656,0.090) (-2.646,0.091) (-2.635,0.091) (-2.624,0.092) (-2.614,0.093) (-2.603,0.094) (-2.593,0.095) (-2.582,0.096) (-2.571,0.097) (-2.561,0.098) (-2.550,0.099) (-2.539,0.100) (-2.529,0.101) (-2.518,0.102) (-2.508,0.102) (-2.497,0.103) (-2.486,0.104) (-2.476,0.105) (-2.465,0.106) (-2.454,0.107) (-2.444,0.108) (-2.433,0.109) (-2.423,0.110) (-2.412,0.111) (-2.401,0.112) (-2.391,0.113) (-2.380,0.114) (-2.369,0.115) (-2.359,0.116) (-2.348,0.117) (-2.338,0.118) (-2.327,0.119) (-2.316,0.121) (-2.306,0.122) (-2.295,0.123) (-2.284,0.124) (-2.274,0.125) (-2.263,0.126) (-2.252,0.127) (-2.242,0.128) (-2.231,0.129) (-2.221,0.130) (-2.210,0.131) (-2.199,0.132) (-2.189,0.133) (-2.178,0.135) (-2.167,0.136) (-2.157,0.137) (-2.146,0.138) (-2.136,0.139) (-2.125,0.140) (-2.114,0.141) (-2.104,0.142) (-2.093,0.144) (-2.083,0.145) (-2.072,0.146) (-2.061,0.147) (-2.051,0.148) (-2.040,0.149) (-2.029,0.151) (-2.019,0.152) (-2.008,0.153) (-1.998,0.154) (-1.987,0.155) (-1.976,0.156) (-1.966,0.158) (-1.955,0.159) (-1.944,0.160) (-1.934,0.161) (-1.923,0.162) (-1.913,0.164) (-1.902,0.165) (-1.891,0.166) (-1.881,0.167) (-1.870,0.168) (-1.859,0.170) (-1.849,0.171) (-1.838,0.172) (-1.828,0.173) (-1.817,0.175) (-1.806,0.176) (-1.796,0.177) (-1.785,0.178) (-1.774,0.180) (-1.764,0.181) (-1.753,0.182) (-1.743,0.183) (-1.732,0.184) (-1.721,0.186) (-1.711,0.187) (-1.700,0.188) (-1.689,0.189) (-1.679,0.191) (-1.668,0.192) (-1.657,0.193) (-1.647,0.194) (-1.636,0.196) (-1.626,0.197) (-1.615,0.198) (-1.604,0.199) (-1.594,0.201) (-1.583,0.202) (-1.573,0.203) (-1.562,0.204) (-1.551,0.206) (-1.541,0.207) (-1.530,0.208) (-1.519,0.209) (-1.509,0.211) (-1.498,0.212) (-1.488,0.213) (-1.477,0.214) (-1.466,0.216) (-1.456,0.217) (-1.445,0.218) (-1.434,0.219) (-1.424,0.221) (-1.413,0.222) (-1.403,0.223) (-1.392,0.224) (-1.381,0.226) (-1.371,0.227) (-1.360,0.228) (-1.349,0.229) (-1.339,0.231) (-1.328,0.232) (-1.318,0.233) (-1.307,0.234) (-1.296,0.235) (-1.286,0.237) (-1.275,0.238) (-1.264,0.239) (-1.254,0.240) (-1.243,0.241) (-1.233,0.243) (-1.222,0.244) (-1.211,0.245) (-1.201,0.246) (-1.190,0.247) (-1.179,0.248) (-1.169,0.250) (-1.158,0.251) (-1.147,0.252) (-1.137,0.253) (-1.126,0.254) (-1.116,0.255) (-1.105,0.257) (-1.094,0.258) (-1.084,0.259) (-1.073,0.260) (-1.062,0.261) (-1.052,0.262) (-1.041,0.263) (-1.031,0.264) (-1.020,0.265) (-1.009,0.266) (-0.999,0.267) (-0.988,0.269) (-0.977,0.270) (-0.967,0.271) (-0.956,0.272) (-0.946,0.273) (-0.935,0.274) (-0.924,0.275) (-0.914,0.276) (-0.903,0.277) (-0.893,0.278) (-0.882,0.279) (-0.871,0.280) (-0.861,0.281) (-0.850,0.282) (-0.839,0.283) (-0.829,0.284) (-0.818,0.285) (-0.808,0.285) (-0.797,0.286) (-0.786,0.287) (-0.776,0.288) (-0.765,0.289) (-0.754,0.290) (-0.744,0.291) (-0.733,0.292) (-0.723,0.293) (-0.712,0.293) (-0.701,0.294) (-0.691,0.295) (-0.680,0.296) (-0.669,0.297) (-0.659,0.298) (-0.648,0.298) (-0.638,0.299) (-0.627,0.300) (-0.616,0.301) (-0.606,0.302) (-0.595,0.303) (-0.584,0.305) (-0.574,0.306) (-0.563,0.308) (-0.552,0.310) (-0.542,0.312) (-0.531,0.316) (-0.521,0.319) (-0.510,0.324) (-0.499,0.330) (-0.489,0.338) (-0.478,0.347) (-0.467,0.358) (-0.457,0.371) (-0.446,0.386) (-0.436,0.405) (-0.425,0.427) (-0.414,0.452) (-0.404,0.481) (-0.393,0.514) (-0.383,0.550) (-0.372,0.590) (-0.361,0.633) (-0.351,0.680) (-0.340,0.729) (-0.329,0.779) (-0.319,0.831) (-0.308,0.883) (-0.298,0.933) (-0.287,0.982) (-0.276,1.027) (-0.266,1.067) (-0.255,1.102) (-0.244,1.130) (-0.234,1.151) (-0.223,1.164) (-0.213,1.168) (-0.202,1.165) (-0.191,1.152) (-0.181,1.132) (-0.170,1.105) (-0.159,1.072) (-0.149,1.032) (-0.138,0.989) (-0.128,0.943) (-0.117,0.895) (-0.106,0.846) (-0.096,0.798) (-0.085,0.752) (-0.074,0.709) (-0.064,0.670) (-0.053,0.635) (-0.043,0.606) (-0.032,0.583) (-0.021,0.566) (-0.011,0.556) (0.000,0.553) (0.011,0.556) (0.021,0.566) (0.032,0.583) (0.043,0.606) (0.053,0.635) (0.064,0.670) (0.074,0.709) (0.085,0.752) (0.096,0.798) (0.106,0.846) (0.117,0.895) (0.128,0.943) (0.138,0.989) (0.149,1.032) (0.159,1.072) (0.170,1.105) (0.181,1.132) (0.191,1.152) (0.202,1.165) (0.212,1.168) (0.223,1.164) (0.234,1.151) (0.244,1.130) (0.255,1.102) (0.266,1.067) (0.276,1.027) (0.287,0.982) (0.297,0.933) (0.308,0.883) (0.319,0.831) (0.329,0.779) (0.340,0.729) (0.351,0.680) (0.361,0.633) (0.372,0.590) (0.382,0.550) (0.393,0.514) (0.404,0.481) (0.414,0.452) (0.425,0.427) (0.436,0.405) (0.446,0.386) (0.457,0.371) (0.467,0.358) (0.478,0.347) (0.489,0.338) (0.499,0.330) (0.510,0.324) (0.521,0.319) (0.531,0.316) (0.542,0.312) (0.552,0.310) (0.563,0.308) (0.574,0.306) (0.584,0.305) (0.595,0.303) (0.606,0.302) (0.616,0.301) (0.627,0.300) (0.638,0.299) (0.648,0.298) (0.659,0.298) (0.669,0.297) (0.680,0.296) (0.691,0.295) (0.701,0.294) (0.712,0.293) (0.722,0.293) (0.733,0.292) (0.744,0.291) (0.754,0.290) (0.765,0.289) (0.776,0.288) (0.786,0.287) (0.797,0.286) (0.807,0.285) (0.818,0.285) (0.829,0.284) (0.839,0.283) (0.850,0.282) (0.861,0.281) (0.871,0.280) (0.882,0.279) (0.892,0.278) (0.903,0.277) (0.914,0.276) (0.924,0.275) (0.935,0.274) (0.946,0.273) (0.956,0.272) (0.967,0.271) (0.977,0.270) (0.988,0.269) (0.999,0.267) (1.009,0.266) (1.020,0.265) (1.031,0.264) (1.041,0.263) (1.052,0.262) (1.062,0.261) (1.073,0.260) (1.084,0.259) (1.094,0.258) (1.105,0.257) (1.116,0.255) (1.126,0.254) (1.137,0.253) (1.147,0.252) (1.158,0.251) (1.169,0.250) (1.179,0.248) (1.190,0.247) (1.201,0.246) (1.211,0.245) (1.222,0.244) (1.233,0.243) (1.243,0.241) (1.254,0.240) (1.264,0.239) (1.275,0.238) (1.286,0.237) (1.296,0.235) (1.307,0.234) (1.318,0.233) (1.328,0.232) (1.339,0.231) (1.349,0.229) (1.360,0.228) (1.371,0.227) (1.381,0.226) (1.392,0.224) (1.403,0.223) (1.413,0.222) (1.424,0.221) (1.434,0.219) (1.445,0.218) (1.456,0.217) (1.466,0.216) (1.477,0.214) (1.488,0.213) (1.498,0.212) (1.509,0.211) (1.519,0.209) (1.530,0.208) (1.541,0.207) (1.551,0.206) (1.562,0.204) (1.572,0.203) (1.583,0.202) (1.594,0.201) (1.604,0.199) (1.615,0.198) (1.626,0.197) (1.636,0.196) (1.647,0.194) (1.657,0.193) (1.668,0.192) (1.679,0.191) (1.689,0.189) (1.700,0.188) (1.711,0.187) (1.721,0.186) (1.732,0.184) (1.742,0.183) (1.753,0.182) (1.764,0.181) (1.774,0.180) (1.785,0.178) (1.796,0.177) (1.806,0.176) (1.817,0.175) (1.827,0.173) (1.838,0.172) (1.849,0.171) (1.859,0.170) (1.870,0.168) (1.881,0.167) (1.891,0.166) (1.902,0.165) (1.912,0.164) (1.923,0.162) (1.934,0.161) (1.944,0.160) (1.955,0.159) (1.966,0.158) (1.976,0.156) (1.987,0.155) (1.997,0.154) (2.008,0.153) (2.019,0.152) (2.029,0.151) (2.040,0.149) (2.051,0.148) (2.061,0.147) (2.072,0.146) (2.083,0.145) (2.093,0.144) (2.104,0.142) (2.114,0.141) (2.125,0.140) (2.136,0.139) (2.146,0.138) (2.157,0.137) (2.167,0.136) (2.178,0.135) (2.189,0.133) (2.199,0.132) (2.210,0.131) (2.221,0.130) (2.231,0.129) (2.242,0.128) (2.252,0.127) (2.263,0.126) (2.274,0.125) (2.284,0.124) (2.295,0.123) (2.306,0.122) (2.316,0.121) (2.327,0.119) (2.338,0.118) (2.348,0.117) (2.359,0.116) (2.369,0.115) (2.380,0.114) (2.391,0.113) (2.401,0.112) (2.412,0.111) (2.423,0.110) (2.433,0.109) (2.444,0.108) (2.454,0.107) (2.465,0.106) (2.476,0.105) (2.486,0.104) (2.497,0.103) (2.508,0.102) (2.518,0.102) (2.529,0.101) (2.539,0.100) (2.550,0.099) (2.561,0.098) (2.571,0.097) (2.582,0.096) (2.592,0.095) (2.603,0.094) (2.614,0.093) (2.624,0.092) (2.635,0.091) (2.646,0.091) (2.656,0.090) (2.667,0.089) (2.677,0.088) (2.688,0.087) (2.699,0.086) (2.709,0.085) (2.720,0.085) (2.731,0.084) (2.741,0.083) (2.752,0.082) (2.762,0.081) (2.773,0.081) (2.784,0.080) (2.794,0.079) (2.805,0.078) (2.816,0.077) (2.826,0.077) (2.837,0.076) (2.847,0.075) (2.858,0.074) (2.869,0.073) (2.879,0.073) (2.890,0.072) (2.901,0.071) (2.911,0.070) (2.922,0.070) (2.932,0.069) (2.943,0.068) (2.954,0.068) (2.964,0.067) (2.975,0.066) (2.986,0.066) (2.996,0.065) (3.007,0.064) (3.018,0.063) (3.028,0.063) (3.039,0.062) (3.049,0.061) (3.060,0.061) (3.071,0.060) (3.081,0.060) (3.092,0.059) (3.103,0.058) (3.113,0.058) (3.124,0.057) (3.134,0.056) (3.145,0.056) (3.156,0.055) (3.166,0.055) (3.177,0.054) (3.188,0.053) (3.198,0.053) (3.209,0.052) (3.219,0.052) (3.230,0.051) (3.241,0.050) (3.251,0.050) (3.262,0.049) (3.272,0.049) (3.283,0.048) (3.294,0.048) (3.304,0.047) (3.315,0.047) (3.326,0.046) (3.336,0.046) (3.347,0.045) (3.357,0.045) (3.368,0.044) (3.379,0.044) (3.389,0.043) (3.400,0.043)};
 
  \fill[gray!35, fill opacity=0.5]
    (-3.400,0) -- plot coordinates {(-3.400,0.042) (-3.389,0.043) (-3.379,0.043) (-3.368,0.044) (-3.357,0.044) (-3.347,0.045) (-3.336,0.045) (-3.326,0.046) (-3.315,0.046) (-3.304,0.047) (-3.294,0.047) (-3.283,0.048) (-3.272,0.048) (-3.262,0.049) (-3.251,0.049) (-3.241,0.050) (-3.230,0.051) (-3.219,0.051) (-3.209,0.052) (-3.198,0.052) (-3.188,0.053) (-3.177,0.053) (-3.166,0.054) (-3.156,0.055) (-3.145,0.055) (-3.134,0.056) (-3.124,0.057) (-3.113,0.057) (-3.103,0.058) (-3.092,0.058) (-3.081,0.059) (-3.071,0.060) (-3.060,0.060) (-3.049,0.061) (-3.039,0.062) (-3.028,0.062) (-3.018,0.063) (-3.007,0.064) (-2.996,0.064) (-2.986,0.065) (-2.975,0.066) (-2.964,0.066) (-2.954,0.067) (-2.943,0.068) (-2.933,0.069) (-2.922,0.069) (-2.911,0.070) (-2.901,0.071) (-2.890,0.071) (-2.879,0.072) (-2.869,0.073) (-2.858,0.074) (-2.848,0.074) (-2.837,0.075) (-2.826,0.076) (-2.816,0.077) (-2.805,0.078) (-2.794,0.078) (-2.784,0.079) (-2.773,0.080) (-2.763,0.081) (-2.752,0.082) (-2.741,0.082) (-2.731,0.083) (-2.720,0.084) (-2.709,0.085) (-2.699,0.086) (-2.688,0.087) (-2.678,0.087) (-2.667,0.088) (-2.656,0.089) (-2.646,0.090) (-2.635,0.091) (-2.624,0.092) (-2.614,0.093) (-2.603,0.094) (-2.593,0.095) (-2.582,0.095) (-2.571,0.096) (-2.561,0.097) (-2.550,0.098) (-2.539,0.099) (-2.529,0.100) (-2.518,0.101) (-2.508,0.102) (-2.497,0.103) (-2.486,0.104) (-2.476,0.105) (-2.465,0.106) (-2.454,0.107) (-2.444,0.108) (-2.433,0.109) (-2.423,0.110) (-2.412,0.111) (-2.401,0.112) (-2.391,0.113) (-2.380,0.114) (-2.369,0.115) (-2.359,0.116) (-2.348,0.117) (-2.338,0.118) (-2.327,0.119) (-2.316,0.120) (-2.306,0.121) (-2.295,0.122) (-2.284,0.123) (-2.274,0.124) (-2.263,0.125) (-2.252,0.126) (-2.242,0.127) (-2.231,0.129) (-2.221,0.130) (-2.210,0.131) (-2.199,0.132) (-2.189,0.133) (-2.178,0.134) (-2.167,0.135) (-2.157,0.136) (-2.146,0.137) (-2.136,0.139) (-2.125,0.140) (-2.114,0.141) (-2.104,0.142) (-2.093,0.143) (-2.083,0.144) (-2.072,0.145) (-2.061,0.147) (-2.051,0.148) (-2.040,0.149) (-2.029,0.150) (-2.019,0.151) (-2.008,0.153) (-1.998,0.154) (-1.987,0.155) (-1.976,0.156) (-1.966,0.157) (-1.955,0.158) (-1.944,0.160) (-1.934,0.161) (-1.923,0.162) (-1.913,0.163) (-1.902,0.165) (-1.891,0.166) (-1.881,0.167) (-1.870,0.168) (-1.859,0.169) (-1.849,0.171) (-1.838,0.172) (-1.828,0.173) (-1.817,0.174) (-1.806,0.176) (-1.796,0.177) (-1.785,0.178) (-1.774,0.179) (-1.764,0.181) (-1.753,0.182) (-1.743,0.183) (-1.732,0.184) (-1.721,0.186) (-1.711,0.187) (-1.700,0.188) (-1.689,0.189) (-1.679,0.191) (-1.668,0.192) (-1.657,0.193) (-1.647,0.194) (-1.636,0.196) (-1.626,0.197) (-1.615,0.198) (-1.604,0.199) (-1.594,0.201) (-1.583,0.202) (-1.573,0.203) (-1.562,0.204) (-1.551,0.206) (-1.541,0.207) (-1.530,0.208) (-1.519,0.209) (-1.509,0.211) (-1.498,0.212) (-1.488,0.213) (-1.477,0.215) (-1.466,0.216) (-1.456,0.217) (-1.445,0.218) (-1.434,0.220) (-1.424,0.221) (-1.413,0.222) (-1.403,0.223) (-1.392,0.225) (-1.381,0.226) (-1.371,0.227) (-1.360,0.228) (-1.349,0.229) (-1.339,0.231) (-1.328,0.232) (-1.318,0.233) (-1.307,0.234) (-1.296,0.236) (-1.286,0.237) (-1.275,0.238) (-1.264,0.239) (-1.254,0.240) (-1.243,0.242) (-1.233,0.243) (-1.222,0.244) (-1.211,0.245) (-1.201,0.246) (-1.190,0.248) (-1.179,0.249) (-1.169,0.250) (-1.158,0.251) (-1.147,0.252) (-1.137,0.253) (-1.126,0.255) (-1.116,0.256) (-1.105,0.257) (-1.094,0.258) (-1.084,0.259) (-1.073,0.260) (-1.062,0.261) (-1.052,0.263) (-1.041,0.264) (-1.031,0.265) (-1.020,0.266) (-1.009,0.267) (-0.999,0.268) (-0.988,0.269) (-0.977,0.270) (-0.967,0.271) (-0.956,0.272) (-0.946,0.273) (-0.935,0.274) (-0.924,0.275) (-0.914,0.276) (-0.903,0.277) (-0.893,0.278) (-0.882,0.279) (-0.871,0.280) (-0.861,0.281) (-0.850,0.282) (-0.839,0.283) (-0.829,0.284) (-0.818,0.285) (-0.808,0.286) (-0.797,0.287) (-0.786,0.288) (-0.776,0.289) (-0.765,0.290) (-0.754,0.291) (-0.744,0.292) (-0.733,0.293) (-0.723,0.293) (-0.712,0.294) (-0.701,0.295) (-0.691,0.296) (-0.680,0.297) (-0.669,0.298) (-0.659,0.298) (-0.648,0.299) (-0.638,0.300) (-0.627,0.301) (-0.616,0.301) (-0.606,0.302) (-0.595,0.303) (-0.584,0.304) (-0.574,0.304) (-0.563,0.305) (-0.552,0.306) (-0.542,0.307) (-0.531,0.307) (-0.521,0.308) (-0.510,0.308) (-0.499,0.309) (-0.489,0.310) (-0.478,0.310) (-0.467,0.311) (-0.457,0.312) (-0.446,0.312) (-0.436,0.313) (-0.425,0.314) (-0.414,0.315) (-0.404,0.316) (-0.393,0.317) (-0.383,0.318) (-0.372,0.319) (-0.361,0.321) (-0.351,0.324) (-0.340,0.327) (-0.329,0.331) (-0.319,0.337) (-0.308,0.344) (-0.298,0.352) (-0.287,0.363) (-0.276,0.377) (-0.266,0.394) (-0.255,0.415) (-0.244,0.441) (-0.234,0.472) (-0.223,0.508) (-0.213,0.551) (-0.202,0.601) (-0.191,0.658) (-0.181,0.722) (-0.170,0.794) (-0.159,0.873) (-0.149,0.959) (-0.138,1.051) (-0.128,1.149) (-0.117,1.249) (-0.106,1.352) (-0.096,1.455) (-0.085,1.555) (-0.074,1.651) (-0.064,1.740) (-0.053,1.820) (-0.043,1.889) (-0.032,1.945) (-0.021,1.986) (-0.011,2.011) (0.000,2.020) (0.011,2.011) (0.021,1.986) (0.032,1.945) (0.043,1.889) (0.053,1.820) (0.064,1.740) (0.074,1.651) (0.085,1.555) (0.096,1.455) (0.106,1.352) (0.117,1.249) (0.128,1.149) (0.138,1.051) (0.149,0.959) (0.159,0.873) (0.170,0.794) (0.181,0.722) (0.191,0.658) (0.202,0.601) (0.212,0.551) (0.223,0.508) (0.234,0.472) (0.244,0.441) (0.255,0.415) (0.266,0.394) (0.276,0.377) (0.287,0.363) (0.297,0.352) (0.308,0.344) (0.319,0.337) (0.329,0.331) (0.340,0.327) (0.351,0.324) (0.361,0.321) (0.372,0.319) (0.382,0.318) (0.393,0.317) (0.404,0.316) (0.414,0.315) (0.425,0.314) (0.436,0.313) (0.446,0.312) (0.457,0.312) (0.467,0.311) (0.478,0.310) (0.489,0.310) (0.499,0.309) (0.510,0.308) (0.521,0.308) (0.531,0.307) (0.542,0.307) (0.552,0.306) (0.563,0.305) (0.574,0.304) (0.584,0.304) (0.595,0.303) (0.606,0.302) (0.616,0.301) (0.627,0.301) (0.638,0.300) (0.648,0.299) (0.659,0.298) (0.669,0.298) (0.680,0.297) (0.691,0.296) (0.701,0.295) (0.712,0.294) (0.722,0.293) (0.733,0.293) (0.744,0.292) (0.754,0.291) (0.765,0.290) (0.776,0.289) (0.786,0.288) (0.797,0.287) (0.807,0.286) (0.818,0.285) (0.829,0.284) (0.839,0.283) (0.850,0.282) (0.861,0.281) (0.871,0.280) (0.882,0.279) (0.892,0.278) (0.903,0.277) (0.914,0.276) (0.924,0.275) (0.935,0.274) (0.946,0.273) (0.956,0.272) (0.967,0.271) (0.977,0.270) (0.988,0.269) (0.999,0.268) (1.009,0.267) (1.020,0.266) (1.031,0.265) (1.041,0.264) (1.052,0.263) (1.062,0.261) (1.073,0.260) (1.084,0.259) (1.094,0.258) (1.105,0.257) (1.116,0.256) (1.126,0.255) (1.137,0.253) (1.147,0.252) (1.158,0.251) (1.169,0.250) (1.179,0.249) (1.190,0.248) (1.201,0.246) (1.211,0.245) (1.222,0.244) (1.233,0.243) (1.243,0.242) (1.254,0.240) (1.264,0.239) (1.275,0.238) (1.286,0.237) (1.296,0.236) (1.307,0.234) (1.318,0.233) (1.328,0.232) (1.339,0.231) (1.349,0.229) (1.360,0.228) (1.371,0.227) (1.381,0.226) (1.392,0.225) (1.403,0.223) (1.413,0.222) (1.424,0.221) (1.434,0.220) (1.445,0.218) (1.456,0.217) (1.466,0.216) (1.477,0.215) (1.488,0.213) (1.498,0.212) (1.509,0.211) (1.519,0.209) (1.530,0.208) (1.541,0.207) (1.551,0.206) (1.562,0.204) (1.572,0.203) (1.583,0.202) (1.594,0.201) (1.604,0.199) (1.615,0.198) (1.626,0.197) (1.636,0.196) (1.647,0.194) (1.657,0.193) (1.668,0.192) (1.679,0.191) (1.689,0.189) (1.700,0.188) (1.711,0.187) (1.721,0.186) (1.732,0.184) (1.742,0.183) (1.753,0.182) (1.764,0.181) (1.774,0.179) (1.785,0.178) (1.796,0.177) (1.806,0.176) (1.817,0.174) (1.827,0.173) (1.838,0.172) (1.849,0.171) (1.859,0.169) (1.870,0.168) (1.881,0.167) (1.891,0.166) (1.902,0.165) (1.912,0.163) (1.923,0.162) (1.934,0.161) (1.944,0.160) (1.955,0.158) (1.966,0.157) (1.976,0.156) (1.987,0.155) (1.997,0.154) (2.008,0.153) (2.019,0.151) (2.029,0.150) (2.040,0.149) (2.051,0.148) (2.061,0.147) (2.072,0.145) (2.083,0.144) (2.093,0.143) (2.104,0.142) (2.114,0.141) (2.125,0.140) (2.136,0.139) (2.146,0.137) (2.157,0.136) (2.167,0.135) (2.178,0.134) (2.189,0.133) (2.199,0.132) (2.210,0.131) (2.221,0.130) (2.231,0.129) (2.242,0.127) (2.252,0.126) (2.263,0.125) (2.274,0.124) (2.284,0.123) (2.295,0.122) (2.306,0.121) (2.316,0.120) (2.327,0.119) (2.338,0.118) (2.348,0.117) (2.359,0.116) (2.369,0.115) (2.380,0.114) (2.391,0.113) (2.401,0.112) (2.412,0.111) (2.423,0.110) (2.433,0.109) (2.444,0.108) (2.454,0.107) (2.465,0.106) (2.476,0.105) (2.486,0.104) (2.497,0.103) (2.508,0.102) (2.518,0.101) (2.529,0.100) (2.539,0.099) (2.550,0.098) (2.561,0.097) (2.571,0.096) (2.582,0.095) (2.592,0.095) (2.603,0.094) (2.614,0.093) (2.624,0.092) (2.635,0.091) (2.646,0.090) (2.656,0.089) (2.667,0.088) (2.677,0.087) (2.688,0.087) (2.699,0.086) (2.709,0.085) (2.720,0.084) (2.731,0.083) (2.741,0.082) (2.752,0.082) (2.762,0.081) (2.773,0.080) (2.784,0.079) (2.794,0.078) (2.805,0.078) (2.816,0.077) (2.826,0.076) (2.837,0.075) (2.847,0.074) (2.858,0.074) (2.869,0.073) (2.879,0.072) (2.890,0.071) (2.901,0.071) (2.911,0.070) (2.922,0.069) (2.932,0.069) (2.943,0.068) (2.954,0.067) (2.964,0.066) (2.975,0.066) (2.986,0.065) (2.996,0.064) (3.007,0.064) (3.018,0.063) (3.028,0.062) (3.039,0.062) (3.049,0.061) (3.060,0.060) (3.071,0.060) (3.081,0.059) (3.092,0.058) (3.103,0.058) (3.113,0.057) (3.124,0.057) (3.134,0.056) (3.145,0.055) (3.156,0.055) (3.166,0.054) (3.177,0.053) (3.188,0.053) (3.198,0.052) (3.209,0.052) (3.219,0.051) (3.230,0.051) (3.241,0.050) (3.251,0.049) (3.262,0.049) (3.272,0.048) (3.283,0.048) (3.294,0.047) (3.304,0.047) (3.315,0.046) (3.326,0.046) (3.336,0.045) (3.347,0.045) (3.357,0.044) (3.368,0.044) (3.379,0.043) (3.389,0.043) (3.400,0.042)} -- (3.400,0) -- cycle;
  \draw[gray!90, thin] plot coordinates {(-3.400,0.042) (-3.389,0.043) (-3.379,0.043) (-3.368,0.044) (-3.357,0.044) (-3.347,0.045) (-3.336,0.045) (-3.326,0.046) (-3.315,0.046) (-3.304,0.047) (-3.294,0.047) (-3.283,0.048) (-3.272,0.048) (-3.262,0.049) (-3.251,0.049) (-3.241,0.050) (-3.230,0.051) (-3.219,0.051) (-3.209,0.052) (-3.198,0.052) (-3.188,0.053) (-3.177,0.053) (-3.166,0.054) (-3.156,0.055) (-3.145,0.055) (-3.134,0.056) (-3.124,0.057) (-3.113,0.057) (-3.103,0.058) (-3.092,0.058) (-3.081,0.059) (-3.071,0.060) (-3.060,0.060) (-3.049,0.061) (-3.039,0.062) (-3.028,0.062) (-3.018,0.063) (-3.007,0.064) (-2.996,0.064) (-2.986,0.065) (-2.975,0.066) (-2.964,0.066) (-2.954,0.067) (-2.943,0.068) (-2.933,0.069) (-2.922,0.069) (-2.911,0.070) (-2.901,0.071) (-2.890,0.071) (-2.879,0.072) (-2.869,0.073) (-2.858,0.074) (-2.848,0.074) (-2.837,0.075) (-2.826,0.076) (-2.816,0.077) (-2.805,0.078) (-2.794,0.078) (-2.784,0.079) (-2.773,0.080) (-2.763,0.081) (-2.752,0.082) (-2.741,0.082) (-2.731,0.083) (-2.720,0.084) (-2.709,0.085) (-2.699,0.086) (-2.688,0.087) (-2.678,0.087) (-2.667,0.088) (-2.656,0.089) (-2.646,0.090) (-2.635,0.091) (-2.624,0.092) (-2.614,0.093) (-2.603,0.094) (-2.593,0.095) (-2.582,0.095) (-2.571,0.096) (-2.561,0.097) (-2.550,0.098) (-2.539,0.099) (-2.529,0.100) (-2.518,0.101) (-2.508,0.102) (-2.497,0.103) (-2.486,0.104) (-2.476,0.105) (-2.465,0.106) (-2.454,0.107) (-2.444,0.108) (-2.433,0.109) (-2.423,0.110) (-2.412,0.111) (-2.401,0.112) (-2.391,0.113) (-2.380,0.114) (-2.369,0.115) (-2.359,0.116) (-2.348,0.117) (-2.338,0.118) (-2.327,0.119) (-2.316,0.120) (-2.306,0.121) (-2.295,0.122) (-2.284,0.123) (-2.274,0.124) (-2.263,0.125) (-2.252,0.126) (-2.242,0.127) (-2.231,0.129) (-2.221,0.130) (-2.210,0.131) (-2.199,0.132) (-2.189,0.133) (-2.178,0.134) (-2.167,0.135) (-2.157,0.136) (-2.146,0.137) (-2.136,0.139) (-2.125,0.140) (-2.114,0.141) (-2.104,0.142) (-2.093,0.143) (-2.083,0.144) (-2.072,0.145) (-2.061,0.147) (-2.051,0.148) (-2.040,0.149) (-2.029,0.150) (-2.019,0.151) (-2.008,0.153) (-1.998,0.154) (-1.987,0.155) (-1.976,0.156) (-1.966,0.157) (-1.955,0.158) (-1.944,0.160) (-1.934,0.161) (-1.923,0.162) (-1.913,0.163) (-1.902,0.165) (-1.891,0.166) (-1.881,0.167) (-1.870,0.168) (-1.859,0.169) (-1.849,0.171) (-1.838,0.172) (-1.828,0.173) (-1.817,0.174) (-1.806,0.176) (-1.796,0.177) (-1.785,0.178) (-1.774,0.179) (-1.764,0.181) (-1.753,0.182) (-1.743,0.183) (-1.732,0.184) (-1.721,0.186) (-1.711,0.187) (-1.700,0.188) (-1.689,0.189) (-1.679,0.191) (-1.668,0.192) (-1.657,0.193) (-1.647,0.194) (-1.636,0.196) (-1.626,0.197) (-1.615,0.198) (-1.604,0.199) (-1.594,0.201) (-1.583,0.202) (-1.573,0.203) (-1.562,0.204) (-1.551,0.206) (-1.541,0.207) (-1.530,0.208) (-1.519,0.209) (-1.509,0.211) (-1.498,0.212) (-1.488,0.213) (-1.477,0.215) (-1.466,0.216) (-1.456,0.217) (-1.445,0.218) (-1.434,0.220) (-1.424,0.221) (-1.413,0.222) (-1.403,0.223) (-1.392,0.225) (-1.381,0.226) (-1.371,0.227) (-1.360,0.228) (-1.349,0.229) (-1.339,0.231) (-1.328,0.232) (-1.318,0.233) (-1.307,0.234) (-1.296,0.236) (-1.286,0.237) (-1.275,0.238) (-1.264,0.239) (-1.254,0.240) (-1.243,0.242) (-1.233,0.243) (-1.222,0.244) (-1.211,0.245) (-1.201,0.246) (-1.190,0.248) (-1.179,0.249) (-1.169,0.250) (-1.158,0.251) (-1.147,0.252) (-1.137,0.253) (-1.126,0.255) (-1.116,0.256) (-1.105,0.257) (-1.094,0.258) (-1.084,0.259) (-1.073,0.260) (-1.062,0.261) (-1.052,0.263) (-1.041,0.264) (-1.031,0.265) (-1.020,0.266) (-1.009,0.267) (-0.999,0.268) (-0.988,0.269) (-0.977,0.270) (-0.967,0.271) (-0.956,0.272) (-0.946,0.273) (-0.935,0.274) (-0.924,0.275) (-0.914,0.276) (-0.903,0.277) (-0.893,0.278) (-0.882,0.279) (-0.871,0.280) (-0.861,0.281) (-0.850,0.282) (-0.839,0.283) (-0.829,0.284) (-0.818,0.285) (-0.808,0.286) (-0.797,0.287) (-0.786,0.288) (-0.776,0.289) (-0.765,0.290) (-0.754,0.291) (-0.744,0.292) (-0.733,0.293) (-0.723,0.293) (-0.712,0.294) (-0.701,0.295) (-0.691,0.296) (-0.680,0.297) (-0.669,0.298) (-0.659,0.298) (-0.648,0.299) (-0.638,0.300) (-0.627,0.301) (-0.616,0.301) (-0.606,0.302) (-0.595,0.303) (-0.584,0.304) (-0.574,0.304) (-0.563,0.305) (-0.552,0.306) (-0.542,0.307) (-0.531,0.307) (-0.521,0.308) (-0.510,0.308) (-0.499,0.309) (-0.489,0.310) (-0.478,0.310) (-0.467,0.311) (-0.457,0.312) (-0.446,0.312) (-0.436,0.313) (-0.425,0.314) (-0.414,0.315) (-0.404,0.316) (-0.393,0.317) (-0.383,0.318) (-0.372,0.319) (-0.361,0.321) (-0.351,0.324) (-0.340,0.327) (-0.329,0.331) (-0.319,0.337) (-0.308,0.344) (-0.298,0.352) (-0.287,0.363) (-0.276,0.377) (-0.266,0.394) (-0.255,0.415) (-0.244,0.441) (-0.234,0.472) (-0.223,0.508) (-0.213,0.551) (-0.202,0.601) (-0.191,0.658) (-0.181,0.722) (-0.170,0.794) (-0.159,0.873) (-0.149,0.959) (-0.138,1.051) (-0.128,1.149) (-0.117,1.249) (-0.106,1.352) (-0.096,1.455) (-0.085,1.555) (-0.074,1.651) (-0.064,1.740) (-0.053,1.820) (-0.043,1.889) (-0.032,1.945) (-0.021,1.986) (-0.011,2.011) (0.000,2.020) (0.011,2.011) (0.021,1.986) (0.032,1.945) (0.043,1.889) (0.053,1.820) (0.064,1.740) (0.074,1.651) (0.085,1.555) (0.096,1.455) (0.106,1.352) (0.117,1.249) (0.128,1.149) (0.138,1.051) (0.149,0.959) (0.159,0.873) (0.170,0.794) (0.181,0.722) (0.191,0.658) (0.202,0.601) (0.212,0.551) (0.223,0.508) (0.234,0.472) (0.244,0.441) (0.255,0.415) (0.266,0.394) (0.276,0.377) (0.287,0.363) (0.297,0.352) (0.308,0.344) (0.319,0.337) (0.329,0.331) (0.340,0.327) (0.351,0.324) (0.361,0.321) (0.372,0.319) (0.382,0.318) (0.393,0.317) (0.404,0.316) (0.414,0.315) (0.425,0.314) (0.436,0.313) (0.446,0.312) (0.457,0.312) (0.467,0.311) (0.478,0.310) (0.489,0.310) (0.499,0.309) (0.510,0.308) (0.521,0.308) (0.531,0.307) (0.542,0.307) (0.552,0.306) (0.563,0.305) (0.574,0.304) (0.584,0.304) (0.595,0.303) (0.606,0.302) (0.616,0.301) (0.627,0.301) (0.638,0.300) (0.648,0.299) (0.659,0.298) (0.669,0.298) (0.680,0.297) (0.691,0.296) (0.701,0.295) (0.712,0.294) (0.722,0.293) (0.733,0.293) (0.744,0.292) (0.754,0.291) (0.765,0.290) (0.776,0.289) (0.786,0.288) (0.797,0.287) (0.807,0.286) (0.818,0.285) (0.829,0.284) (0.839,0.283) (0.850,0.282) (0.861,0.281) (0.871,0.280) (0.882,0.279) (0.892,0.278) (0.903,0.277) (0.914,0.276) (0.924,0.275) (0.935,0.274) (0.946,0.273) (0.956,0.272) (0.967,0.271) (0.977,0.270) (0.988,0.269) (0.999,0.268) (1.009,0.267) (1.020,0.266) (1.031,0.265) (1.041,0.264) (1.052,0.263) (1.062,0.261) (1.073,0.260) (1.084,0.259) (1.094,0.258) (1.105,0.257) (1.116,0.256) (1.126,0.255) (1.137,0.253) (1.147,0.252) (1.158,0.251) (1.169,0.250) (1.179,0.249) (1.190,0.248) (1.201,0.246) (1.211,0.245) (1.222,0.244) (1.233,0.243) (1.243,0.242) (1.254,0.240) (1.264,0.239) (1.275,0.238) (1.286,0.237) (1.296,0.236) (1.307,0.234) (1.318,0.233) (1.328,0.232) (1.339,0.231) (1.349,0.229) (1.360,0.228) (1.371,0.227) (1.381,0.226) (1.392,0.225) (1.403,0.223) (1.413,0.222) (1.424,0.221) (1.434,0.220) (1.445,0.218) (1.456,0.217) (1.466,0.216) (1.477,0.215) (1.488,0.213) (1.498,0.212) (1.509,0.211) (1.519,0.209) (1.530,0.208) (1.541,0.207) (1.551,0.206) (1.562,0.204) (1.572,0.203) (1.583,0.202) (1.594,0.201) (1.604,0.199) (1.615,0.198) (1.626,0.197) (1.636,0.196) (1.647,0.194) (1.657,0.193) (1.668,0.192) (1.679,0.191) (1.689,0.189) (1.700,0.188) (1.711,0.187) (1.721,0.186) (1.732,0.184) (1.742,0.183) (1.753,0.182) (1.764,0.181) (1.774,0.179) (1.785,0.178) (1.796,0.177) (1.806,0.176) (1.817,0.174) (1.827,0.173) (1.838,0.172) (1.849,0.171) (1.859,0.169) (1.870,0.168) (1.881,0.167) (1.891,0.166) (1.902,0.165) (1.912,0.163) (1.923,0.162) (1.934,0.161) (1.944,0.160) (1.955,0.158) (1.966,0.157) (1.976,0.156) (1.987,0.155) (1.997,0.154) (2.008,0.153) (2.019,0.151) (2.029,0.150) (2.040,0.149) (2.051,0.148) (2.061,0.147) (2.072,0.145) (2.083,0.144) (2.093,0.143) (2.104,0.142) (2.114,0.141) (2.125,0.140) (2.136,0.139) (2.146,0.137) (2.157,0.136) (2.167,0.135) (2.178,0.134) (2.189,0.133) (2.199,0.132) (2.210,0.131) (2.221,0.130) (2.231,0.129) (2.242,0.127) (2.252,0.126) (2.263,0.125) (2.274,0.124) (2.284,0.123) (2.295,0.122) (2.306,0.121) (2.316,0.120) (2.327,0.119) (2.338,0.118) (2.348,0.117) (2.359,0.116) (2.369,0.115) (2.380,0.114) (2.391,0.113) (2.401,0.112) (2.412,0.111) (2.423,0.110) (2.433,0.109) (2.444,0.108) (2.454,0.107) (2.465,0.106) (2.476,0.105) (2.486,0.104) (2.497,0.103) (2.508,0.102) (2.518,0.101) (2.529,0.100) (2.539,0.099) (2.550,0.098) (2.561,0.097) (2.571,0.096) (2.582,0.095) (2.592,0.095) (2.603,0.094) (2.614,0.093) (2.624,0.092) (2.635,0.091) (2.646,0.090) (2.656,0.089) (2.667,0.088) (2.677,0.087) (2.688,0.087) (2.699,0.086) (2.709,0.085) (2.720,0.084) (2.731,0.083) (2.741,0.082) (2.752,0.082) (2.762,0.081) (2.773,0.080) (2.784,0.079) (2.794,0.078) (2.805,0.078) (2.816,0.077) (2.826,0.076) (2.837,0.075) (2.847,0.074) (2.858,0.074) (2.869,0.073) (2.879,0.072) (2.890,0.071) (2.901,0.071) (2.911,0.070) (2.922,0.069) (2.932,0.069) (2.943,0.068) (2.954,0.067) (2.964,0.066) (2.975,0.066) (2.986,0.065) (2.996,0.064) (3.007,0.064) (3.018,0.063) (3.028,0.062) (3.039,0.062) (3.049,0.061) (3.060,0.060) (3.071,0.060) (3.081,0.059) (3.092,0.058) (3.103,0.058) (3.113,0.057) (3.124,0.057) (3.134,0.056) (3.145,0.055) (3.156,0.055) (3.166,0.054) (3.177,0.053) (3.188,0.053) (3.198,0.052) (3.209,0.052) (3.219,0.051) (3.230,0.051) (3.241,0.050) (3.251,0.049) (3.262,0.049) (3.272,0.048) (3.283,0.048) (3.294,0.047) (3.304,0.047) (3.315,0.046) (3.326,0.046) (3.336,0.045) (3.347,0.045) (3.357,0.044) (3.368,0.044) (3.379,0.043) (3.389,0.043) (3.400,0.042)};
 
  \node at (0.75,1.95) {$\theta_{\#}\mu$};
  \node at (0.95,1.15) {$\theta_{\#}\nu$};
\end{scope}
 
\draw[->, very thin] (-0.2,-3.3) to[bend right=40] (1.7,-5.6);
\node[left] at (-0.4,-4.3) {$x \mapsto \langle x, \theta\rangle$};
 
\draw[->, very thin] (8.4,-3.3) to[bend left=40] (6.5,-5.6);
\node[right] at (8.6,-4.3) {$x \mapsto \langle x, \theta\rangle$};
 
\end{tikzpicture}